\documentclass[10pt]{article}
\usepackage{amsmath}
\usepackage{amsfonts}
\usepackage{graphicx}
\usepackage{xcolor}

\newtheorem{theorem}{Theorem}
\newtheorem{proposition}[theorem]{Proposition}
\newenvironment{proof}[1][Proof]{\noindent\textbf{#1.} }{\ \rule{0.5em}{0.5em}}

\newtheorem{lem}{Lemma}
\newtheorem{defn}{Definition}
\newtheorem{rem}{Remark}

\begin{document}

\title{{\huge Numerical approximation of the averaged controllability for the wave equation with unknown velocity of propagation}}
\author{Mouna Abdelli\thanks{Laboratory of Mathematics Informatics and Systems (LAMIS), University of Larbi Tebessi, 12002 Tebessa, Algeria (mouna.abdelli@univ-tebessa.dz)} \; and \; Carlos Castro\thanks{Departamento de Matemática e Informática, ETSI Caminos, Canales y Puertos, Polytecnical University of Madrid, 28040 Madrid, Spain (carlos.castro@upm.es)}}
\date{}
\maketitle

\textbf{Abstract: We propose a numerical method to approximate the exact averaged boundary control of a family of wave equations depending on an unknown parameter $\sigma$. More precisely the control, independent of $\sigma$, that drives an initial data to a family of final states at time $t=T$, whose average in $\sigma$ is given. The idea is to project the control problem in the finite dimensional space generated by the first $N$ eigenfunctions of the Laplace operator. {When applied to a single (nonparametric) wave equation, the resulting discrete control problem turns out to be equivalent to the Galerkin approximation proposed by F. Bourquin et al. in \cite{BBL}}. We give a convergence result of the discrete controls to the continuous one. {The method is illustrated with several examples in 1-d and 2-d in a square domain and allows us to give some conjectures on the averaged controllability for the continuous problem.}}



\section{Introduction}

Consider the wave equation with a missing parameter $\sigma $ and
a control $f$ acting on one part of the boundary:%
\begin{equation}
\left\{ 
\begin{array}{l}
u_{tt}-a(\sigma )\Delta u=0 \\ 
u=f\chi _{\Gamma_{0}} \\ 
u(x,0)=u^{0},\quad u_{t}(x,0)=u^{1}%
\end{array}%
\begin{array}{c}
in\text{ }Q \\ 
on\text{ }\Sigma \\ 
in\text{ }\Omega%
\end{array}%
\right.  \label{eq.1}
\end{equation}%
where $\Omega $ is an open bounded domain in $\mathbb{R}^{d}$ with smooth
boundary $\Gamma $ $,\Gamma_{0}$ be an open nonempty subset of $\Gamma ,$ $\chi_{\Gamma_0}$ the characteristic function of the set $\Gamma_0$, $%
t\in \left[ 0,T\right] ,T>0,Q=\Omega \times (0,T),a(\sigma )\in L^\infty$ is the square of the
velocity of propagation and we assume
$$
0<a_m\leq a(\sigma)\leq a_M<\infty, \quad  \sigma \in \Upsilon \subset \mathbb{R},
$$   
for some constants $a_m,a_M$ and $\Upsilon$ an interval of $\mathbb{R}$. 
The function $f=f(x,t)$ is a control, independent of the unknown parameter $\sigma$ and which acts on $\Gamma _{0}$. 

For each value of the parameter $\sigma\in \Upsilon$, problem (\ref{eq.1}) has a unique solution $(u,u_{t})\in C(\left[ 0,T\right]
,L^{2}\left( \Omega \right) \times H^{-1}\left( \Omega \right) )$ for any $%
f\in L^{2}\left( (0,T)\times \Gamma_{0}\right) $ and $(u^{0},u^{1})\in
L^{2}\left( \Omega \right) \times H^{-1}\left( \Omega \right) $. Of course, these solutions will
depend on the parameter $\sigma \in \Upsilon$ and we will write $%
u(x,t;\sigma)$, when we want to make this dependence explicit.

Moreover, there exists $C=C(T)>0$, independent of $\sigma\in \Upsilon$, such
that the solution of (\ref{eq.1}) satisfies the following inequality%
\begin{equation}  \label{eq.2}
\left\Vert (u,u_{t})\right\Vert _{L^{\infty }(\left[ 0,T\right] ;L^{2}\left(
\Omega \right) \times H^{-1}\left( \Omega \right) )}\leq C\left( \left\Vert
(u^{0},u^{1})\right\Vert _{L^{2}\left( \Omega \right) \times H^{-1}\left(
\Omega \right) }+\left\Vert f\right\Vert _{L^{2}\left( (0,T)\times \Gamma
_{0}\right) } \right).
\end{equation}

We are interested in the numerical approximation of the following controllability problem: Given $%
T>0$, the initial data $(u^0,u^1) \in L^2(\Omega)\times H^{-1}(\Omega)$ and a final
target $(u^0_T,u^1_T) \in L^2(\Omega)\times H^{-1}(\Omega)$, find $f\in
L^2(0,T;\Gamma_0)$ such that the solution of (\ref{eq.1}) satisfies 
\begin{equation}  \label{eq.3}
\int_\Upsilon u(x,T;\sigma) d\sigma =u^0_T,\quad \int_\Upsilon
u_t(x,T;\sigma) d\sigma =u^1_T, \quad x\in \Omega.
\end{equation}
When such a control exists we say that the initial data $(u^{0},u^{1})\in
L^{2}\left( \Omega \right) \times H^{-1}\left( \Omega \right) $ is  \textbf{%
controllable in average} to the target $(u^0_T,u^1_T)$.

This notion of controllability in average (or averaged controllability) was first introduced by E. Zuazua in \cite{Zua1}, where sufficient conditions were given for the controllability of finite dimensional systems. In this work the author also highlight the difficulty to extend classical controllability results for PDE to obtain averaged controls. Precisely, the fact that the dynamics of the average is no longer solution of the same PDE. There have been a considerable effort to overcome this difficulty and obtain results in this direction for different distributed systems (see for example \cite{re30}, \cite{re40}). For the particular case of the wave equation, we can address to \cite{re24} where the authors consider two different wave equation with the same control, or \cite{LohZ} where a more general family of wave equations is considered. More precisely, they deal with a parametric family  of wave equations where the parameter is a measure perturbation of a Dirac mass. 

{
Control problems for parameter families of partial differential equations have attracted a considerable interest in the last years due to their applicability in engineering processes, where uncertainty in one or several parameters of the model is common. In this case, it is natural to look for controls valid for any value of the unknown parameter $\sigma\in \Upsilon$ and therefore, independent of $\sigma$. In general, such controls are not likely to produce solutions which attain the same target for all values of the parameters, as for instance in our case $(u(x,T;\sigma), u_t(x,T;\sigma))=(u^0_T,u^1_T)$ for all values of $\sigma$. Instead, a weakened objective must be considered. Here we focus on the notion of control in average for which the target is the average in $\sigma$ of the solutions at time $t=T$. 
The controllability in average is only one of the multiple ways to deal with an unknown parameter. There are other notions in the literature according to specific applications as robust control, where one looks for an approximate control valid for a range of parameters, or the risk averse control where the control tries to avoid some undesirable behaviors in the solutions. We refer to \cite{Periago} and the references therein where such problems are considered in the context of the optimal control.  }

To illustrate the mathematical difficulties behind the particular situation we consider in this work, i.e. the control in average for the wave equation, we mention that for the one-dimensional model with the simplest choice $a(\sigma)=\sigma,$ $\sigma\in [\sigma_0,\sigma_1]$, the existence of controls is still open. A partial result is given in \cite{LohZ} assuming, among other properties, that $\sigma_0^{-1}\sigma_1\notin \mathbb{Q},$ which does not seem a natural hypothesis in this context. 

In this paper we do not address this difficult problem. We rather focus on the numerical approximation of the averaged controls, assuming that they exist. However, this is also far from being a simple exercise since the usual approach based on the discretization of the controlled wave equation via finite differences or finite elements does not inherit the properties of the continuous model. Even for a single wave equation (i.e. without a parameter dependence), the discrete controls obtained for the associated finite dimensional systems may become unbounded as the discretization parameter goes to zero. This was first observed in \cite{gll} where the authors considered a finite dimensional version of the Hilbert Uniqueness Method introduced in \cite{Li}. Since then, several cures have been proposed to recover convergence approximations of the controls as bigrid algorithms, Tychonoff regularization, filtering, mixed finite elements, etc (see for instance \cite{GLH}, \cite{CMM}, \cite{EZ}). We refer to the review \cite{ZuaN} for a detailed description of such problems and references. 

{More recently, a direct approximation of the  optimality system associated to the continuous controllability problem with a conformal finite element method has proven to be convergent in quite general situations as general geometries, nonconstant coefficients and systems (see \cite{CMu}, \cite{CFM}). In this approach, a suitable Lagrangian formulation is introduced where the controlled and adjoint equation are imposed as a constraint with a suitable Lagrange multiplier. Following a similar  Lagrangian approach, a convergence result for nonconformal finite elements has been recently obtained in \cite{BFO}.}  

{
Here we follow a different approach based on a projection method of the control problem on the finite dimensional subspace constituted by the first eigenfunctions of the Laplacian. Then, we apply the Hilbert Uniqueness Method to characterize the minimal $L^2$-norm controls. It turns out that the variational formulation of the resulting finite dimensional control problem can be interpreted as a Galerkin approximation for the continuous one. 
For nonparametric PDE's such Galerkin approximation was proposed by F. Bourquin and co-workers in the nineties to approximate the boundary control (see \cite{BBL}, \cite{B1}, \cite{B2}). Here we adapt it to obtain averaged controls.}  As in the nonparametric case, the existence and convergence of controls can be deduced by the usual theory of Galerkin approximations for variational problems. The finite dimensional control problem can be solved using the finite dimensional theory introduced in \cite{Zua1} which provides an explicit expression for the control in average. 

The main advantage of this approach is that it provides a simpler and efficient method to obtain
numerical approximations of controls without any regularization technique. Moreover, it can be easily adapted to parametric systems as (\ref{eq.1}) that require to compute a large number of control problems to approximate the averages involved. In fact, it can be used in any situation where the control can be characterized variationally (simultaneous control, parametric regional control \cite{AH}, etc.).  An important drawback is that it requires to compute the eigenfunctions of the Laplace operator. {This is simple in one-dimensional  problems or two dimensional ones in special domains (rectangles or disks), but not for general domains or variable coefficients equations. For general domains, there is an added difficulty in the approximation of the normal derivatives in the variational characterization of the controls (see formula (\ref{eq.cara_dis_con}) below). In \cite{BBL} a second projection is proposed for this approximation, this time on the subspace generated by the eigenvectors of the associated capacitance operator on the boundary. We give an alternative method that avoids to approximate this eigenvalue problem.} 

{To illustrate the method we present some numerical experiments in one dimension and two dimensions in a square domain. We also give numerical evidences of several unknown properties for the dynamics of the average and the control in average. In particular we find the following:
\begin{enumerate}
\item The wave equation is controllable in average, with no restrictions on the parameters interval, as long as the time $T$ is sufficiently large. In particular this is true when the wave equation is controllable for all values of $\sigma \in \Upsilon$, but it is also true even when this property holds only for a subinterval in $\Upsilon$.  
\item The dynamics of the average decays in time, without control, to a constant which depends on the initial data. When we apply a control, it takes advantage of this decay and acts mainly at the end of the time interval.
\item The averaged control converges to the control of the averaged parameter when the parameter belongs to an interval of length $\varepsilon \to 0$.
\end{enumerate}
}

The rest of the paper is divided as follows: in section \ref{sec2} we adapt the usual variational characterization of controls for the wave equation to the control in average. In section \ref{sec3} we introduce the observability inequality that allows to prove the existence of controls and a particular class of controls that is obtained by minimization of a suitable functional. In section \ref{sec4} we introduce the numerical method. In section \ref{sec5} we prove the convergence of discrete controls to the continuous one. In section \ref{sec6} we deduce matrix formulation equivalent to the discrete system and the controllability of the finite dimensional system. Finally, in section \ref{sec7} we present some numerical examples in one dimension and in two dimensions on a square domain.

\section{Variational characterization of the control in average}
\label{sec2}

In this section we introduce a variational characterization of the controls in average that we use later to find their numerical approximation. In particular we prove that a class of controls in average can be obtained as minimizers of a quadratic functional defined on a Hilbert space. The results of this section are not new and can be found in \cite{re30} (appendix A.2). We include them here for completeness. 

For technical reasons we restrict ourselves to controls which are zero near $t=0,T$. This affects to the quadratic functional that we define below. We follow the approach in \cite{MiZu} for the wave equation, that we adapt to the parametric system (\ref{eq.1}) for the control in average.  

Let us consider the following backwards wave equation, for each $\sigma \in \Upsilon$, 
\begin{equation}
\left\{ 
\begin{array}{l}
\phi _{tt}-a(\sigma )\Delta \phi =0 \\ 
\phi =0 \\ 
\phi (x,T)=\phi ^{0},\varphi _{t}(x,T)=\phi ^{1}%
\end{array}%
\begin{array}{c}
in\text{ }Q \\ 
on\text{ }\Sigma \\ 
in\text{ }\Omega%
\end{array}%
\right.  \label{eq.6}
\end{equation}%
where $(\phi ^{0},\phi ^{1})\in H_{0}^{1}\left( \Omega \right) \times
L^{2}\left( \Omega \right) $ are independent of $\sigma $.

We also define the duality product between $L^{2}(\Omega)\times H^{-1}(\Omega )$ and 
$H_{0}^{1}(\Omega)\times L^{2}(\Omega )$ by 
\begin{equation*}
\left\langle (\phi ^{0},\phi ^{1}),(u^{0},u^{1})\right\rangle =\int_{\Omega
}u^{0}\phi ^{1}d\gamma-\left\langle u^{1},\phi ^{0}\right\rangle _{-1,1},
\end{equation*}%
where $\left\langle \cdot ,\cdot \right\rangle _{-1,1}$ is the usual duality
product between $H_{0}^{1}$ and its dual space $H^{-1}$.

The following result provides a variational characterization of the control in average.

\begin{lem}
\label{le.1} Assume that for $T>0$ and the data  $(u^0,u^1), (u^{0}_T,u^1_T) \in L^2\times H^{-1}$ there exists a control function $f\in L^{2}\left(
(0,T)\times \Gamma _{0}\right) $ for which the solutions of (\ref{eq.1}) satisfy (\ref{eq.3}). Then, $f$ is solution of the variational
identity 
\begin{eqnarray}
&&\int_{0}^{T}\int_{\Gamma _{0}}\int_{\Upsilon }a(\sigma)\frac{\partial \phi }{%
\partial \upsilon }d\sigma f d\gamma dt-\left\langle \left( \int_{\Upsilon }\phi
(\cdot ,0;\sigma )d\sigma ,\int_{\Upsilon }\phi _{t}(\cdot ,0;\sigma
)d\sigma \right) ,(u^{0},u^{1})\right\rangle  \notag \\
&&+\left\langle (\phi ^{0},\phi ^{1}),(u^0_T,u^1_T)\right\rangle =0,
\label{eq.7}
\end{eqnarray}%
for all $(\phi ^{0},\phi ^{1})\in H_{0}^{1}\times L^{2}(\Omega )$, where $%
(\phi ,\phi _{t})$ is the solution of the backwards wave equation (\ref{eq.6}%
). Reciprocally, if $f$ satisfies (\ref{eq.7}) then (\ref{eq.3}) holds.
\end{lem}

\begin{proof}
Let us first suppose that $(u^{0},u^{1}),(\phi ^{0},\phi ^{1})\in \mathcal{D}%
(\Omega )\times \mathcal{D}(\Omega ),f\in $ $\mathcal{D}((0,T),\Gamma _{0})$
and let $u$ and $\phi $ be the smooth solutions of (\ref{eq.1}) and (\ref%
{eq.6}) respectively.

Multiplying the equation of $u$ by $\phi $ and integrating by parts one
obtains%
\begin{eqnarray} \nonumber
0 &=&\int_{0}^{T}\int_{\Omega }\int_{\Upsilon }\left( u_{tt}-a(\sigma
)\Delta u\right) \phi d\sigma dxdt=\left. \int_{\Omega }\int_{\Upsilon
}\left( \phi u_{t}-\phi _{t}u\right) d\sigma dx\right\vert
_{0}^{T}\\ \nonumber
&&+\int_{0}^{T}\int_{\Gamma }\int_{\Upsilon } a(\sigma)\left( -\frac{\partial u}{%
\partial \upsilon }\phi  +\frac{\partial \phi }{\partial \upsilon }u\right)
 d\sigma d\gamma dt \\ \nonumber
&=&\int_{0}^{T}\int_{\Gamma _{0}}\int_{\Upsilon }a(\sigma)\frac{\partial \phi }{%
\partial \upsilon }d\sigma f d\gamma dt+\int_{\Omega }\left(\phi
^{0} \int_{\Upsilon }u_t(T;\sigma)d\sigma -\phi ^{1}\int_{\Upsilon }u(T;\sigma)d\sigma \right) dx 
\\ \label{eq.vari_ca}
&&-\int_{\Omega }\left( u^{0} \int_{\Upsilon }\phi
\left( 0;\sigma \right) d\sigma -u^{1} \int_{\Upsilon}\phi _{t}(0;\sigma) d\sigma  \right)  dx
\end{eqnarray}
By a density argument we deduce that an analogous formula holds for any $(u^{0},u^{1})\in L^{2}\left(
\Omega \right) \times H^{-1}\left( \Omega \right) $, $(\phi^{0},\phi
^{1})\in H_{0}^{1}\times L^{2}(\Omega )$ and $f\in L^2(0,T;\Gamma_0)$. We only have to replace the integrals by duality products when corresponding.  Now, if $f$ is a control such that (\ref{eq.3}) holds then the weak form of (\ref{eq.vari_ca}) is equivalent to  
(\ref{eq.7}). Reciprocally, if (\ref{eq.7}) holds then the weak form of (\ref{eq.vari_ca}) is equivalent to 
$$
\left\langle (\phi^{0},\phi^{1}),\left( \int_{\Upsilon }u_t(T;\sigma)d\sigma, \int_{\Upsilon }u(T;\sigma)d\sigma \right) \right\rangle
=\left\langle (\phi^{0},\phi^{1}),\left( u^{T,0}, u^{T,1} \right) \right\rangle,
$$
for all $(\phi ^{0},\phi^{1})\in H_{0}^{1}(\Omega)\times L^{2}(\Omega ),$ and this is equivalent to (\ref{eq.3}).
\end{proof}

One possibility to construct controls $f$ that satisfy the variational condition (\ref{eq.7}) is as minimizers of a particular quadratic functional. We define the following cost functional $%
J:H_{0}^{1}\left( \Omega \right) \times L^{2}\left( \Omega \right)
\rightarrow 
\mathbb{R}
$ by:%
\begin{eqnarray}
J(\phi ^{0},\phi ^{1}) &=&\frac{1}{2}\int_{0}^{T}\eta (t)\int_{\Gamma
_{0}}\left\vert \int_{\Upsilon }a(\sigma)\frac{\partial \phi }{\partial \upsilon }%
d\sigma \right\vert ^{2}d\gamma dt-\left\langle \left( \int_{\Upsilon }\phi (\cdot
,0;\sigma )d\sigma ,\int_{\Upsilon }\phi _{t}(\cdot ,0;\sigma )d\sigma
\right) ,(u^{0},u^{1})\right\rangle  \notag   \\ \label{eq.8}
&&+\left\langle (\phi ^{0},\phi ^{1}),(u^0_T,u^1_T)\right\rangle ,
\end{eqnarray}%
where $(\phi ,\phi _{t})$ is the solution of (\ref{eq.6}) with the final
data $(\phi ^{0},\phi ^{1})\in H_{0}^{1}\left( \Omega \right) \times
L^{2}\left( \Omega \right) $. The function $\eta (t)$ is a prescribed smooth function in $[0,T]$ introduced to guarantee that the controls vanish in a neighborhood of $t=0,T$. Thus, we consider $\delta>0$ arbitrarily small and,
\begin{equation} \label{eq.eta}
\eta (t)=\left\{ 
\begin{array}{c}
1 \\ 
0%
\end{array}%
\begin{array}{l}
in\text{ }\left[ \delta ,T-\delta \right] \\ 
in\text{ }t=[0,\delta/2]\cup[T+\delta/2,T].%
\end{array}%
\right.
\end{equation} 
The function $\eta$ will depend on $\delta>0$ but we do not make explicit this dependence in the notation to simplify.

\begin{theorem} \label{th_1}
Let $(u^{0},u^{1}),(u^0_T,u^1_T)\in L^{2}\left( \Omega \right) \times H^{-1}\left( \Omega
\right) $ and suppose that $(\hat{\phi}^{0},\hat{\phi}^{1})\in
H_{0}^{1}\left( \Omega \right) \times L^{2}\left( \Omega \right) $ is a
minimizer of $J$ . If $\hat{\phi}$ is the corresponding solution of (\ref%
{eq.6}) with final data $(\hat{\phi}^{0},\hat{\phi}^{1})$ then 
\begin{equation}
f(x,t)=\eta (t)\int_{\Upsilon } a(\sigma)\frac{\partial \hat{\phi}}{\partial \upsilon }%
_{|\Gamma _{0}}d\sigma ,
\end{equation}%
is a control such that the solution of (\ref{eq.1}) satisfies (\ref{eq.3}).
\end{theorem}

\begin{proof}
The Gateaux derivative ot $J$ at $(\hat{\phi}^{0},\hat{\phi}^{1})$ in the direction $({\phi}^{0},{\phi}^{1})$ is given by
\begin{eqnarray*}
&&\lim_{h\rightarrow 0}\frac{1}{h}J((\hat{\phi}^{0},\hat{\phi}^{1})+h(\phi
^{0},\phi ^{1}))-J(\hat{\phi}^{0},\hat{\phi}^{1}) \\
&=&\int_{0}^{T}\eta (t)\int_{\Gamma _{0}}\int_{\Upsilon} a(\sigma)
\frac{\partial \hat{\phi}}{\partial \upsilon }d\sigma \int_{\Upsilon}a(\sigma)\frac{\partial \phi }{\partial \upsilon }d\sigma
d\gamma dt-\left\langle \left( \int_{\Upsilon }\phi (\cdot ,0;\sigma )d\sigma
,\int_{\Upsilon }\phi _{t}(\cdot ,0;\sigma )d\sigma \right)
,(u^{0},u^{1})\right\rangle \\
&&+\left\langle (\phi ^{0},\phi ^{1}),(u^0_T,u^1_T)\right\rangle .
\end{eqnarray*}%
If $J$ achieves its minimum at $(\hat{\phi}^{0},\hat{\phi}^{1})$ we have 
\begin{eqnarray*}
0&=&
\int_{0}^{T}\eta (t)\int_{\Gamma _{0}}\int_{\Upsilon}a(\sigma)\frac{%
\partial \hat{\phi}}{\partial \upsilon }d\sigma \int_{\Upsilon}a(\sigma)
\frac{\partial \phi }{\partial \upsilon }d\sigma d\gamma dt-\left\langle
\left( \int_{\Upsilon }\phi (\cdot ,0;\sigma )d\sigma ,\int_{\Upsilon }\phi
_{t}(\cdot ,0;\sigma )d\sigma \right) ,(u^{0},u^{1})\right\rangle
\\
&&+\left\langle (\phi ^{0},\phi ^{1}),(u^0_T,u^1_T)\right\rangle ,
\end{eqnarray*}
for all $(\phi^0,\phi^1)\in H^1_0(\Omega)\times L^2(\Omega)$. 

From Lemma \ref{le.1} it follows that $f=\eta (t)\int_{\Upsilon}a(\sigma)\dfrac{\partial \hat{\phi}}{\partial \upsilon }_{|\Gamma
_{0}}d\sigma $ is a control for which (\ref{eq.3}) holds.
\end{proof}

\section{The observability inequality in average} \label{sec3}

Let us now give a general condition which ensures the existence of a minimizer for $J$ and therefore a control in average for system (\ref{eq.1}).

\begin{defn} \label{def1}
The system (\ref{eq.6}) is \textbf{observable in average} in time $T>0$ if for some $\varepsilon>0$  there
exists a constant $C_{1}>0$ such that 
\begin{equation} 
C_{1}\left\Vert (\phi ^{0},\phi ^{1})\right\Vert _{H_{0}^{1}\left( \Omega
\right) \times L^{2}\left( \Omega \right) }^{2}\leq \int_{\varepsilon}^{T-\varepsilon}\int_{\Gamma
_{0}}\left\vert \int_{\Upsilon }a(\sigma)\frac{\partial \phi }{\partial \upsilon }%
d\sigma \right\vert ^{2}d\gamma dt  \label{eq.10}
\end{equation}%
for any $(\phi ^{0},\phi ^{1})\in H_{0}^{1}\left( \Omega \right) \times
L^{2}\left( \Omega \right) $ where $\phi $ is the solution of (\ref{eq.6})
with final data $(\phi ^{0},\phi ^{1}).$
\end{defn}

In particular, when system (\ref{eq.6}) is \textbf{observable in average} in time $T>0$ we can choose $\eta (t)$ as in (\ref{eq.eta}), with $\delta>0$ sufficiently small, in such a way that
\begin{equation} \label{eq.10eta}
C_{1}\left\Vert (\phi ^{0},\phi ^{1})\right\Vert _{H_{0}^{1}\left( \Omega
\right) \times L^{2}\left( \Omega \right) }^{2}\leq \int_{0}^{T
}\int_{\Gamma _{0}}\eta (t)\left\vert \int_{\Upsilon }a(\sigma)\frac{\partial \phi }{%
\partial \upsilon }d\sigma \right\vert ^{2}d\gamma dt .
\end{equation}
It is enough to consider $\delta = \varepsilon$ where $\varepsilon$ is the constant in Definition \ref{def1}, since in this way $\eta(t)=1 $ for $t\in [\varepsilon, T-\varepsilon]$ and formula (\ref{eq.10eta}) is a direct consequence of (\ref{eq.10}).

From now on, when a system is observable in average we will assume that $\eta$ is chosen in such a way that (\ref{eq.10eta}) holds. 

\begin{rem}
The parameter $\varepsilon>0$ in the previous definition is necessary to deduce (\ref{eq.10eta}) from (\ref{eq.10}). If we consider $\varepsilon=0$, we obtain a more natural version of the observability inequality but now it is not clear how to deduce  (\ref{eq.10eta}). This is in contrast with a single wave equation (i.e. without parameter dependence) where (\ref{eq.10eta}) and (\ref{eq.10eta}) are equivalent, for sufficiently large time $T$. 
\end{rem}

We prove below that the observability in average that we have defined above is a sufficient condition for the controllability in average of system (\ref{eq.1}). But we first state a technical lemma
that will be used later.

\begin{lem}
\label{le.ene} There exists a constant $C>0$, independent of $\sigma \in
\Upsilon $, such that 
\begin{equation} \label{eq.10b}
\left\Vert \left( \int_{\Upsilon }\phi (\cdot ,0;\sigma )d\sigma
,\int_{\Upsilon }\phi _{t}(\cdot ,0;\sigma )d\sigma \right) \right\Vert
_{H_{0}^{1}(\Omega)\times L^{2}(\Omega)}\leq C\Vert (\phi ^{0},\phi ^{1})\Vert
_{H_{0}^{1}(\Omega)\times L^{2}(\Omega)},
\end{equation}%
for all solutions $\phi $ of the adjoint system (\ref{eq.6}) with final data 
$(\phi ^{0},\phi ^{1})\in H_{0}^{1}(\Omega)\times L^{2}(\Omega )$.
\end{lem}

\begin{proof}
Observe that,
\begin{eqnarray*}
&& \left\Vert \left( \int_{\Upsilon }\phi (\cdot ,0;\sigma )d\sigma
,\int_{\Upsilon }\phi _{t}(\cdot ,0;\sigma )d\sigma \right) \right\Vert
_{H_{0}^{1}\times L^{2}} \leq \int_{\Upsilon }\left\Vert \left( \phi
(\cdot ,0;\sigma ),\phi _{t}(\cdot ,0;\sigma )\right) \right\Vert
_{H_{0}^{1}\times L^{2}}d\sigma \\
&& \quad \leq \int_{\Upsilon }\left( \frac{a(\sigma )}{a_{m}}\left\Vert \phi (\cdot
,0;\sigma )\right\Vert _{H_{0}^{1}}^{2}+\left\Vert \phi _{t}(\cdot ,0;\sigma
)\right\Vert _{L^{2}}^{2}\right) ^{\frac{1}{2}}d\sigma \\
&& \quad \leq \left( \frac{1}{a_{m}}+1\right) ^{\frac{1}{2}}\int_{\Upsilon }\left(
a(\sigma )\left\Vert \phi (\cdot ,0;\sigma )\right\Vert
_{H_{0}^{1}}^{2}+\left\Vert \phi _{t}(\cdot ,0;\sigma )\right\Vert
_{L^{2}}^{2}\right) ^{\frac{1}{2}}d\sigma \\
&&\quad =\left( \frac{1}{a_{m}}+1\right) ^{\frac{1}{2}}\int_{\Upsilon }\left(
a(\sigma )\left\Vert \phi ^{0}\right\Vert _{H_{0}^{1}}^{2}+\left\Vert \phi
^{1}\right\Vert _{L^{2}}^{2}\right) ^{\frac{1}{2}}d\sigma 
\leq \left( \frac{1}{a_{m}}+1\right) ^{\frac{1}{2}}\left( a_{M}+1\right) ^{%
\frac{1}{2}}\left( \left\Vert \phi ^{0}\right\Vert
_{H_{0}^{1}}^{2}+\left\Vert \phi ^{1}\right\Vert _{L^{2}}^{2}\right) ^{\frac{%
1}{2}} .
\end{eqnarray*}%
Then, inequality (\ref{eq.10b}) holds with $C=\left( \dfrac{1}{a_{m}}+1\right) ^{\frac{1}{2}}\left(
a_{M}+1\right) ^{\frac{1}{2}}$ and $a_{m},a_{M}$ are the minimum and maximum
value of $a(\sigma )$, respectively. 
\end{proof}

\begin{theorem} \label{th_2}
If the system (\ref{eq.6}) is \textbf{observable in average} then the functional $%
J $ defined by (\ref{eq.8}) has an unique minimizer $(\hat{\phi}^{0},\hat{%
\phi}^{1})\in H_{0}^{1}\left( \Omega \right) \times L^{2}\left( \Omega
\right) $
\end{theorem}

\begin{proof}
It is easy to see that $J$ is continuous and convex. The continuity of the second term in (\ref{eq.8}) is a consequence of Lemma \ref{le.ene} while for the first term one just observe that  
\begin{eqnarray} \notag
&& \left\| \eta (t) \int_{\Upsilon }a(\sigma)\frac{\partial \phi }{
\partial \upsilon }d\sigma \right\|_{L^2((0,T)\times \Gamma_0)} 
\leq \int_{\Upsilon } \left\| \eta (t) a(\sigma)\frac{\partial \phi }{
\partial \upsilon } \right\|_{L^2((0,T)\times \Gamma_0)} d\sigma 
\\ \label{eq:dir_in}
&& \quad \leq \int_{\Upsilon } a(\sigma) C(\sigma,T) \left\| (\phi^0,\phi^1) \right\|_{H^1_0(\Omega)\times L^2(\Omega)} d\sigma \leq C' \left\| (\phi^0,\phi^1) \right\|_{H^1_0(\Omega)\times L^2(\Omega)},
\end{eqnarray}
where $C(\sigma,T)$ is the constant in the classical direct (regularity) inequality for the constant coefficients wave equation (\ref{eq.1}). The dependence on $\sigma $ of this constant is easily obtained from its proof (see \cite{Li}). In particular, $C(\sigma,T)=C(1+T\sigma)/\sigma$ that is integrable in $\Upsilon$, and the continuity of this first term in $J$ holds.

The existence of a
minimum is ensured if we prove that $J$ is also coercive i.e.%
\begin{equation} \label{eq.coerc}
\lim_{\left\Vert (\phi ^{0},\phi ^{1})\right\Vert _{H_{0}^{1}\left( \Omega
\right) \times L^{2}\left( \Omega \right) }\rightarrow \infty }J(\phi
^{0},\phi ^{1})=\infty
\end{equation}

The coercivity of the functional $J$ follows immediately from (\ref{eq.10b}) and (\ref{eq.10eta}).
Indeed,%
\begin{eqnarray*}
J(\phi ^{0},\phi ^{1}) &\geq &\frac{1}{2}\int_{0}^{T}\eta (t)\int_{\Gamma
_{0}}\left\vert \int_{\Upsilon }a(\sigma)\frac{\partial \phi }{\partial \upsilon }%
d\sigma \right\vert ^{2}d\gamma dt-\left\Vert (\phi ^{0},\phi ^{1})\right\Vert
_{H_{0}^{1}\left( \Omega \right) \times L^{2}\left( \Omega \right)
}\left\Vert (u^0_T,u^1_T)\right\Vert _{L^{2}\left( \Omega \right) \times
H^{-1}\left( \Omega \right) } \\
&&-\left\Vert \left( \int_{\Upsilon }\phi (\cdot ,0;\sigma )d\sigma
,\int_{\Upsilon }\phi _{t}(\cdot ,0;\sigma )d\sigma \right) \right\Vert
_{H_{0}^{1}\left( \Omega \right) \times L^{2}\left( \Omega \right)
}\left\Vert (u^{0},u^{1})\right\Vert _{L^{2}\left( \Omega \right) \times
H^{-1}\left( \Omega \right) } \\
&\geq &C_{1}\left\Vert (\phi ^{0},\phi ^{1})\right\Vert _{H_{0}^{1}\left(
\Omega \right) \times L^{2}\left( \Omega \right) }^{2}-\left\Vert (\phi
^{0},\phi ^{1})\right\Vert _{H_{0}^{1}\left( \Omega \right) \times
L^{2}\left( \Omega \right) }\left\Vert (u^0_T,u^1_T)\right\Vert
_{L^{2}\left( \Omega \right) \times H^{-1}\left( \Omega \right) } \\
&&-C\Vert (\phi ^{0},\phi ^{1})\Vert _{H_{0}^{1}\times L^{2}}\left\Vert
(u^{0},u^{1})\right\Vert _{L^{2}\left( \Omega \right) \times H^{-1}\left(
\Omega \right) } .
\end{eqnarray*}
Condition (\ref{eq.coerc}) follows inmediatly from this inequality.
\end{proof}

We now give two properties of the control obtained by minimization of the functional $J$.

\begin{proposition}
\label{le.pp1} For some initial data $(u^0,u^1)\in L^2(\Omega) \times H^{-1}(\Omega)$, let $f(t)=\eta (t)\int_{\Upsilon }a(\sigma)\dfrac{\partial \hat \phi }{\partial
\upsilon }_{|\Gamma _{0}}d\sigma $ be the averaged control given by
the minimizer $(\hat \phi^0,\hat \phi^1)$ of the functional $J$. If $g$ $\in L^{2}\left( (0,T)\times \Gamma
_{0}\right) $ is any other control driving to zero the average of the state
of system (\ref{eq.1}) then
\begin{equation} \label{eq.11}
\int_0^T  \int_{\Gamma_0}|f|^2 ds\; \frac{dt}{\eta(t)} \leq
\int_0^T \int_{\Gamma_0}|g|^2 ds\; \frac{dt}{\eta(t)}
\end{equation}
\end{proposition}

\begin{rem}
This result establishes that the control given from the minimizer of $J$ is the one that minimizes the $L^2-$weighted norm associated to the measure $\frac1{\eta(t)} dt$. 
\end{rem}

\begin{proof}
Let $(\hat{\phi}^{0},\hat{\phi}^{1})$ be the minimizer for the functional $J$ .
Consider now relation (\ref{eq.7}) for the control $\eta (t)\int_{\Upsilon }a(\sigma)
\frac{\partial \hat \phi }{\partial \upsilon }_{|\Gamma _{0}}d\sigma $ . By
taking $(\hat{\phi}^{0},\hat{\phi}^{1})$ as test function we obtain that%
\begin{eqnarray*}
\int_0^T  \int_{\Gamma_0}|f|^2 ds\; \frac{dt}{\eta(t)}
&=&\int_{0}^{T}\int_{\Gamma _{0}}\eta (t)\left\vert \int_{\Upsilon }a(\sigma)\frac{%
\partial \hat{\phi}}{\partial \upsilon }d\sigma \right\vert ^{2}d\gamma dt \\
&=&\left\langle \left( \int_{\Upsilon }\hat{\phi}(\cdot ,0;\sigma
)d\sigma ,\int_{\Upsilon }\hat{\phi}_{t}(\cdot ,0;\sigma )d\sigma \right)
,(u^{0},u^{1})\right\rangle -\left\langle (\hat{\phi}^{0},\hat{\phi}%
^{1}),(u^0_T,u^1_T)\right\rangle
\end{eqnarray*}

On the other hand, relation (\ref{eq.7}) for the control $g$ and test
function $(\hat{\phi}^{0},\hat{\phi}^{1})$ gives%
\begin{equation*}
\int_{0}^{T}\int_{\Gamma _{0}}\int_{\Upsilon }a(\sigma)\frac{\partial \hat \phi }{\partial
\upsilon }d\sigma g(x,t)\; d\gamma\; dt=\left\langle \left( \int_{\Upsilon }\hat{\phi}(\cdot
,0;\sigma )d\sigma ,\int_{\Upsilon }\hat{\phi}_{t}(\cdot ,0;\sigma )d\sigma
\right) ,(u^{0},u^{1})\right\rangle -\left\langle (\hat{\phi}^{0},\hat{\phi}%
^{1}),(u^0_T,u^1_T)\right\rangle
\end{equation*}

Combining the last two identities we obtain that%
\begin{eqnarray*}
\int_0^T  \int_{\Gamma_0}|f|^2 ds\; \frac{dt}{\eta(t)}
&=&\int_{0}^{T}\int_{\Gamma _{0}}\int_{\Upsilon }\eta(t)a(\sigma)\frac{\partial \hat \phi 
}{\partial \upsilon }d\sigma g\; ds\;\frac{dt}{\eta(t)} \\
&\leq &\left(\int_0^T  \int_{\Gamma_0}\left| \int_\Upsilon \eta(t) a(\sigma)\frac{\partial \hat \phi 
}{\partial \upsilon } d\sigma \right|^2 ds\; \frac{dt}{\eta(t)} \right)^{1/2}\left(
\int_0^T \int_{\Gamma_0}|g|^2 ds\; \frac{dt}{\eta(t)} \right)^{1/2}
.
\end{eqnarray*}%
Dividing this inequality by $\left(\int_0^T  \int_{\Gamma_0}|f|^2 ds\; \frac{dt}{\eta(t)}\right)^{1/2}$ we easily obtain (\ref{eq.11}).
\end{proof}

The next result provides a bound of the control from the initial and target data.

\begin{proposition} \label{prop.4}
Let $f(t)=\eta (t)\int_{\Upsilon }a(\sigma)\dfrac{\partial \hat{\phi}}{\partial
\upsilon }_{|\Gamma _{0}}d\sigma $ be the averaged control, associated to the
initial data $(u^{0},u^{1})$ and target $(u^{0}_T,u^{1}_T)$, obtained by minimization of the functional $J$. 
Then there exists a constant, independent of $%
\sigma \in \Upsilon $, such that 
\begin{equation}
\int_{0}^{T}\frac{1}{\eta (t)}\left\vert f\left( t\right) \right\vert
^{2}dt\leq C\left( \left\Vert (u^{0},u^{1})\right\Vert _{L^{2}\left( \Omega
\right) \times H^{-1}\left( \Omega \right) }^{2}+\left\Vert
(u^0_T,u^1_T)\right\Vert _{L^{2}\left( \Omega \right) \times
H^{-1}\left( \Omega \right) }^{2}\right)  \label{eq.prop4}
\end{equation}
\end{proposition}

\begin{proof}
As $\hat{\phi}$ is the solution of the adjoint system (\ref{eq.6})
associated to the the minimizer $(\hat{\phi}^{0},\hat{\phi}^{1})$ of $J$ in $%
H_{0}^{1}\left( \Omega \right) \times L^{2}\left( \Omega \right) $, we have
in particular 
\begin{equation}
J(\hat{\phi}^{0},\hat{\phi}^{1})\leq J(0,0).
\end{equation}%
Then, taking into account Lemma \ref{le.ene}, we have 
\begin{eqnarray*}
&&\frac{1}{2}\int_{0}^{T}\int_{\Gamma _{0}}\eta (t)\left\vert \int_{\Upsilon
}a(\sigma )\frac{\partial \hat{\phi}}{\partial \upsilon }d\sigma \right\vert
^{2}d\gamma dt\\
& \leq& \left\langle (\hat{\phi}^{0},\hat{\phi}^{1}),(u^0_T,u^1_T)%
\right\rangle -\left\langle \left( \int_{\Upsilon }\phi (\cdot ,0;\sigma
)d\sigma ,\int_{\Upsilon }\phi _{t}(\cdot ,0;\sigma )d\sigma \right)
,(u^{0},u^{1})\right\rangle \\
&\leq &\left\Vert (\hat{\phi}^{0},\hat{\phi}^{1})\right\Vert
_{H_{0}^{1}\left( \Omega \right) \times L^{2}\left( \Omega \right)
}\left\Vert (u^0_T,u^1_T)\right\Vert _{L^{2}\left( \Omega \right) \times
H^{-1}\left( \Omega \right) } \\
&&+\left\Vert \left( \int_{\Upsilon }\phi (\cdot ,0;\sigma )d\sigma
,\int_{\Upsilon }\phi _{t}(\cdot ,0;\sigma )d\sigma \right) \right\Vert
_{H_{0}^{1}\left( \Omega \right) \times L^{2}\left( \Omega \right)
}\left\Vert (u^{0},u^{1})\right\Vert _{L^{2}\left( \Omega \right) \times
H^{-1}\left( \Omega \right) } \\
&\leq &C\left\Vert (\hat{\phi}^{0},\hat{\phi}^{1})\right\Vert
_{H_{0}^{1}\left( \Omega \right) \times L^{2}\left( \Omega \right) }\left(
\left\Vert (u^{0},u^{1})\right\Vert _{L^{2}\left( \Omega \right) \times
H^{-1}\left( \Omega \right) }+\left\Vert (u^0_T,u^1_T)\right\Vert
_{L^{2}\left( \Omega \right) \times H^{-1}\left( \Omega \right) }\right) \\
&\leq &\frac{C}{\sqrt{C_{1}}}\left( \int_{0}^{T}\int_{\Gamma _{0}}\eta
(t)\left\vert \int_{\Upsilon}a(\sigma)\frac{\partial \phi }{\partial
\upsilon }d\sigma \right\vert ^{2}d\gamma dt\right) ^{\frac{1}{2}}\left(
\left\Vert (u^{0},u^{1})\right\Vert _{L^{2}\left( \Omega \right) \times
H^{-1}\left( \Omega \right) }+\left\Vert (u^0_T,u^1_T)\right\Vert
_{L^{2}\left( \Omega \right) \times H^{-1}\left( \Omega \right) }\right)
\end{eqnarray*}%
Dividing by $\left( \int_{0}^{T}\int_{\Gamma _{0}}\eta (t)\left\vert
\int_{\Upsilon}a(\sigma)\dfrac{\partial \phi }{\partial \upsilon }%
d\sigma \right\vert ^{2}d\gamma dt\right) ^{\frac{1}{2}}$ we easily obtain (\ref%
{eq.prop4}).
\end{proof}

\section{Numerical approximation of the control problem} \label{sec4}

In this section we introduce a suitable discretization of the control
problem (\ref{eq.1}) that we use later to find numerical approximations of
the controls. {Unlike previous works (see \cite{BBL}, \cite{B1}), where the projection method is applied directly to the variational formulation of the control problem (\ref{eq.7}), we follow a more natural strategy which consists in applying the projection operator to the original control system (\ref{eq.1}). In this way, we obtain a discrete control problem whose controls approximate the continuous one. In the next section we prove that both strategies are equivalent.}

Let $\{ w_k\}_{k\in \mathbb{N}}$ be an orthonormal basis of $L^2(\Omega)$
constituted by eigenfunctions of the Laplace operator with Dirichlet
boundary conditions, $-\Delta_D$, and $\{\lambda_k^2\}_{k\in \mathbb{N}}$
the corresponding eigenvalues. We assume that they are ordered increasingly,
i.e. 
$
0<\lambda_1^2<\lambda_2^2 \leq \lambda_3^2 \leq,...
$
We define the scaled spaces $\mathcal{H}^\alpha(\Omega)$ as 
\begin{equation*}
\mathcal{H}^\alpha =\left\{ \sum_{k=1}^\infty c_k w_k(x), \quad \mbox{ with }
\sum_{k=1}^\infty \lambda_k^{2\alpha} |c_k|^{2} <\infty \right\}.
\end{equation*}
Note that $\mathcal{H}^0=L^2(\Omega)$, $\mathcal{H}^1=H^1_0(\Omega)$ and $\mathcal{H}^{-1}=H^{-1}(\Omega)$.

We also introduce the finite dimensional space $X^N$ generated by the first $%
N$ eigenfunctions, i.e. 
\begin{equation*}
X^N=\left\{ \sum_{k=1}^N c_k w_k(x), \quad \mbox{ with }
c_k\in \mathbb{R} \right\} \subset \mathcal{H}^\alpha,\quad \alpha=-1,0,1.
\end{equation*}
and the projection operator 
\begin{equation*}
P^N:L^2(\Omega) \to X^N,
\end{equation*}
defined by 
\begin{equation*}
P^N(u)=\sum_{k=1}^N u_k w_k(x), \quad \mbox{ for } u=\sum_{k=1}^\infty u_k
w_k(x).
\end{equation*}
Clearly, this operator can be extended to the analogous projection in $%
\mathcal{H}^{\alpha} $ for $\alpha=-1,1$, that we still denote by $P^N$.

The idea is to project system (\ref{eq.1}) into the finite dimensional space 
$X^N$, but we first transform this system to introduce a more suitable
representation of the boundary control.

For each $t\in [0,T]$ we introduce the following elliptic problem with the control at the boundary,
\begin{equation}
\left\{ 
\begin{array}{ll}
\Delta h=0, & x\in \Omega, \\ 
h=f \chi _{\Gamma _{0}}, & x\in \Gamma_0, \\ 
h=0, & x\in \Gamma \backslash \Gamma_0 .%
\end{array}%
\right.
\end{equation}

Then, we have at least formally that $v=\Delta _{D}^{-1}u$ satisfies the
system 
\begin{equation}
\left\{ 
\begin{array}{l}
v_{tt}-a(\sigma )\Delta v=-a(\sigma)h(x,t) \\ 
v=0 \\ 
v(x,0)=v^{0},\qquad v_{t}(x,0)=v^{1},%
\end{array}%
\begin{array}{c}
in\text{ }Q \\ 
on\text{ }\Sigma \\ 
in\text{ }\Omega%
\end{array}%
\right.  \label{eq.v}
\end{equation}%
where $(v^{0},v^{1})=(\Delta _{D}^{-1}u^{0},\Delta _{D}^{-1}u^{1})$. Now, we
apply the projection operator to system (\ref{eq.v}). Taking into account
that the Laplace operator conmmutes with $P^{N}$ and leave invariant the
subspace $X^{N}$ we easily obtain the finite dimensional system%
\begin{equation}
\left\{ 
\begin{array}{l}
v_{tt}^{N}-a(\sigma )D^{N}v^{N}=-a(\sigma)P^{N}h(x,t) \\ 
v^{N}\in X^{N} \\ 
v^{N}(x,0)=v^{0,N},v_{t}^{N}(x,0)=v^{1,N}%
\end{array}%
\begin{array}{c}
in\text{ }Q \\ 
for\text{ }t\in \lbrack 0,t] \\ 
in\text{ }\Omega%
\end{array}%
\right.
\end{equation}%
where $(v^{0,N},v^{1,N})=(P^{N}\Delta _{D}^{-1}u^{0},P^{N}\Delta
_{D}^{-1}u^{1})$ and $D^{N}=P^{N}\Delta :X^{N}\rightarrow X^{N}$. Finally,
we apply the operator $D^{N}$, and we obtain for the new variable $%
u^{N}=D^{N}v^{N}$ the system 
\begin{equation}
\left\{ 
\begin{array}{l}
u_{tt}^{N}-a(\sigma )D^{N}u^{N}=-a(\sigma)D^{N}P^{N}h(x,t),\\ 
u^{N}\in X^{N}, \\ 
u^{N}(x,0)=u^{0,N},u_{t}^{N}(x,0)=u^{1,N},%
\end{array}%
\begin{array}{c}
in\text{ }Q, \\ 
for\text{ }t\in \lbrack 0,T], \\ 
in\text{ }\Omega,%
\end{array}%
\right.  \label{eq.20}
\end{equation}%
\begin{equation}
\left\{ 
\begin{array}{ll}
\Delta h=0, & x\in \Omega , \\ 
h=f\chi _{\Gamma _{0}}, & x\in \Gamma _{0}, \\ 
h=0, & x\in \Gamma \backslash \Gamma _{0}.%
\end{array}%
\right.  \label{eq.21}
\end{equation}%
where $(u^{0,N},u^{1,N})=(P^{N}u^{0},P^{N}u^{1})$.

We adopt (\ref{eq.20})-(\ref{eq.21}) as the discrete approximation of the
control system (\ref{eq.1}). More precisely, the discrete control problem
reads: Given $T>0$, $(u^{0,N},u^{1,N})\in X^{N}\times X^N$ and a target $%
(u^{0,N}_T,u^{1,N}_T) \in X^{N}\times X^{N}$, find a control $f^N\in
L^2(0,T;\Gamma_0)$ such that the solution of (\ref{eq.20})-(\ref{eq.21})
with $f=f^N$ satisfies 
\begin{equation}  \label{eq.22}
\int_\Upsilon u^N(T;\sigma) d\sigma =u^{0,N}_T,\quad \int_\Upsilon
u_t^N(T;\sigma) d\sigma =u^{1,N}_T.
\end{equation}

Note that the elliptic problem in (\ref{eq.21}) is not discretized in the
above formulation. For one dimensional problems this is not necessary since $%
h$ can be computed explicitly from the control $f$. However, in more
general problems a further discretization of the controls and this elliptic problem would be also
required. {In \cite{BBL} the authors propose a second projection of the controls in the subspace of $L^2(\Gamma_0)$ constituted by the eigenfunctions of the associated capacitance operator in the boundary. However, this requires to compute or approximate numerically this new eigenvalue problem. A simpler alternative consists in parametrizing the boundary $\Gamma_0$ (assuming some smoothness on $\Omega$) and then use a projection on the trigonometric basis associated to the parameter interval. More precisely, if we are in dimension $d=2$ and assume (to simplify) that $\Gamma_0$ can be parametrized by a single invertible $C^1$-function ${\bf r}: (0,1) \to \Gamma_0$, with $\| {\bf r}'(s)\|>\alpha>0$ for all $s\in(0,1)$, then any function $g\in L^2(\Gamma_0)$ can be associated with another function $\sqrt{\| {\bf r}'(s)\|} g({\bf r}(s))\in L^2(0,1)$. Now we can project this function on the finite dimensional space generated by the first functions of the trigonometric basis $\{\sin (j\pi s)\}_{j=1}^{M}$, $M>0 $ given, and we obtain the following discretization of $g\in L^2(\Gamma_0)$,
\begin{equation} \label{eq:bj}
g^{M}(x)=\sum_{j=1}^{M} g_j \beta_j(x), \quad g_j=2\int_0^1 \sqrt{\| {\bf r}'(s)\|} g({\bf r}(s)) \sin(j\pi s) \; ds, \quad 
\beta_j(x)=\frac{\sin(j\pi{\bf r}^{-1}(x))}{\sqrt{\| {\bf r}'({\bf r}^{-1}(x))\|}}.
\end{equation}
The control $f^N$ is then approximated by 
\begin{equation} \label{eq:ap_com}
f^{N,M}(x,t)=\sum_{j=1}^M f^{N,M}_j(t)\beta_j(x), \quad x\in \Gamma_0. 
\end{equation}
Now, for each $j=1,...,M$ we define $h_j(x)$ the solution of (\ref{eq.21}) with boundary data $\beta_j(x)$ in $\Gamma_0$, 
and replace the discrete control problem (\ref{eq.20})-(\ref{eq.21}) by 
\begin{equation}
\left\{ 
\begin{array}{l}
u_{tt}^{N}-a(\sigma )D^{N}u^{N}=-a(\sigma)\sum_{j=1}^Mf_j^{N,M}(t)D^{N}P^{N}h_j(x),\\ 
u^{N}\in X^{N}, \\ 
u^{N}(x,0)=u^{0,N},u_{t}^{N}(x,0)=u^{1,N},%
\end{array}%
\begin{array}{c}
in\text{ }Q, \\ 
for\text{ }t\in \lbrack 0,T], \\ 
in\text{ }\Omega.%
\end{array}%
\right.  \label{eq.20v2}
\end{equation}%
Here the controls are the $M$ functions $f_j^{N,M}(t)$, $j=1,...,M$. Note that in general the number of controls $M$ is  much smaller than the dimension of the control problem $N$. Observe also that one can divide $\Gamma_0$ in several curves and use a different parametrization for each one. The discrete problem has a similar structure in this case but the number of controls will be larger.}

\section{Convergence of discrete averaged controls}
\label{sec5}

We first introduce a variational characterization of the discrete controls following the same approach as in the continuous system. Let us consider the following backwards wave equation:
\begin{equation}
\left\{ 
\begin{array}{l}
\phi _{tt}^{N}-a(\sigma )D^{N}\phi ^{N}=0 \\ 
\phi ^{N}\in X^{N} \\ 
\phi ^{N}(x,T)=\phi ^{0,N},\phi _{t}^{N}(x,T)=\phi ^{1,N}%
\end{array}%
\begin{array}{c}
in\text{ }Q \\ 
for\text{ }t\in \lbrack 0,t] \\ 
in\text{ }\Omega%
\end{array}%
\right.  \label{eq.23}
\end{equation}%
where $(\varphi ^{0,N},\varphi ^{1,N})\in X^{N}\times X^N$. 

Note that, due to our discretization scheme, the solution $\phi^N$ of system (\ref{eq.23}) coincides with the solution of the continuous system (\ref{eq.6}) with the initial data $(\varphi ^{0,N},\varphi ^{1,N})\in X^{N}\times X^N$.

We also introduce the following scalar product in $X^N\times X^N$ 
\begin{equation} \label{eq_peN}
\left\langle (\phi ^{0,N},\phi ^{1,N}),(u^{0,N},u^{1,N})\right\rangle_N 
=\left\langle (\phi ^{0,N},\phi ^{1,N}),(u^{0,N},u^{1,N})\right\rangle,
\end{equation}
that coincides with the duality product 
between $L^{2}(\Omega)\times H^{-1}(\Omega )$ and 
$H_{0}^{1}(\Omega)\times L^{2}(\Omega )$ for functions in $X^N\times X^N$. 

The following result is the analogous to Lemma \ref{le.1} for the discrete control problem. 

\begin{lem}\label{le.1disc}
Assume that for $T>0$ and the data $\left( u^{0,N},u^{1,N}\right),\left( u^{0,N}_T,u^{1,N}_T\right) \in X^{N}\times X^{N}$ the 
control $f^{N}\in
L^{2}(0,T;\Gamma_0)$ makes the solution of the discrete system (\ref{eq.20})-(\ref{eq.21}) to satisfy (\ref{eq.22}). Then,
\begin{eqnarray} \nonumber
&& \int_{0}^{T}\int_{\Gamma_0 }\int_{\Upsilon }a(\sigma)\frac{\partial \phi^N}{\partial \upsilon}d\sigma f^{N}(x,t) d\gamma dt-\left\langle \left( \int_{\Upsilon }\phi^N
(\cdot ,0;\sigma )d\sigma ,\int_{\Upsilon }\phi^N _{t}(\cdot ,0;\sigma
)d\sigma \right) ,(u^{0,N},u^{1,N})\right\rangle_N \\
&&+\left\langle (\phi
^{0,N},\phi ^{1,N}),(u^{0,N}_T,u^{1,N}_T)\right\rangle_N =0 
\label{eq.cara_dis_con}
\end{eqnarray}
for all $(\phi^{N,0},\phi ^{N,1}) \in X^N\times X^N$, where $(\phi^N,\phi^N_t)$ is the solution of the discrete backwards system (\ref{eq.23}). 
\end{lem}

\begin{proof}
Following the proof of Lemma \ref{le.1} we multiply the equation of $u^N$ in (\ref{eq.20}) by the solution of the adjoint problem $\phi^N$ and integrate by parts. We easily obtain
\begin{eqnarray*}
0&=&-\int_{0}^{T}\int_{\Omega } \int_{\Upsilon }a(\sigma)\phi
^{N}d\sigma D^{N}P^{N} h^N(x,t) dxdt-\left\langle \left( \int_{\Upsilon }\phi^N
(\cdot ,0;\sigma )d\sigma ,\int_{\Upsilon }\phi^N _{t}(\cdot ,0;\sigma
)d\sigma \right) ,(u^{0},u^{1})\right\rangle_N\\
&&+\left\langle (\phi
^{N,0},\phi ^{N,1}),(u^0_T,u^1_T)\right\rangle_N 
\end{eqnarray*}
where $h^N$ is the solution of (\ref{eq.21}) with $f=f^N$ the averaged control. 

Therefore, it is enough to prove
\begin{equation*}
\int_{0}^{T}\int_{\Gamma _{0}} \int_{\Upsilon }a(\sigma)\frac{\partial \phi
^{N}}{\partial \upsilon }d\sigma f^Nd\gamma dt=-\int_{0}^{T}\int_{\Omega } \int_{\Upsilon }a(\sigma)\phi
^{N}D^{N}P^{N}d\sigma h^{N}(x,t) \; dxdt.
\end{equation*}

Note that 
\begin{equation*}
\phi ^{N}\left( x,t,\sigma \right) =\sum_{j=1}^{N }\phi
_{j}^N(t;\sigma)w_{j}(x)
\end{equation*}%
where $w_{j}(x),\lambda _{j}$ are the eigenfunctions and eigenvalues
of the operator $-\Delta .$

Using Green formula we easily obtain,
\begin{eqnarray*}
\int_{\Gamma _{0}}\frac{\partial w_{j}}{\partial \upsilon }f^{N}d\gamma
&=&\int_{\Gamma _{0}}\frac{\partial w_{j}}{\partial \upsilon }h^{N}d\gamma \\
&=&\int_{\Omega }\Delta w_{j}h^{N}dx-\int_{\Omega }w_{j}\Delta
h^{N}dx+\int_{\Gamma _{0}}w_{j}\frac{\partial h^{N}}{\partial \upsilon }dx \\
&=&-\int_{\Omega }\lambda _{j}w_{j}h^{N}dx
\end{eqnarray*}

Then,
\begin{eqnarray*}
\int_{0}^{T}\int_{\Gamma _{0}} \int_{\Upsilon }a(\sigma)\frac{\partial \phi
^{N}}{\partial \upsilon }d\sigma f^N dxdt
&=&\sum_{j=1}^N\int_{0}^{T}\int_{\Gamma _{0}}\int_{\Upsilon }a(\sigma)\phi_j^N\frac{\partial w_j}{%
\partial \upsilon }d\sigma f^{N}dxdt \\
&=&-\int_{0}^{T}\sum_{j=1}^{N }\int_{\Upsilon }a(\sigma)\phi _{j}^N(t;\sigma)\int_{\Omega
}\lambda _{j}w_{j}h^{N}d\sigma dxdt\\
&=&-\int_{0}^{T}\int_{\Omega
}\int_{\Upsilon }a(\sigma)D^N\phi^N d\sigma h^{N} dxdt .
\end{eqnarray*}
This concludes the proof.
\end{proof}

\begin{rem}
As the solutions of systems (\ref{eq.6}) and (\ref{eq.23}) coincide,  the variational characterization of the discrete controls in (\ref{eq.cara_dis_con}) is analogous to the one associated to the continuous control problem, but for functions in the subset $X^N\times X^N\subset H^1_0(\Omega)\times L^2(\Omega)$. This means in particular, that any control of the continuous system is also a control for the discrete one. 

Of course, the reciprocal is not true. In the rest of the paper we construct a sequence of discrete controls that converges to the control that minimizes $J$, assuming that the continuous system is observable in average. 
\end{rem}

We define the following quadratic cost functional $J^{N}:X^{N}\times
X^{N}\rightarrow 
\mathbb{R}
$ by:%
\begin{eqnarray}
J^N(\phi ^{N,0},\phi ^{N,1}) &=&\frac{1}{2}\int_{0}^{T}\int_{\Gamma_0 }\eta
(t)\left\vert \int_{\Upsilon }a(\sigma)\frac{\partial \varphi ^{N}}{\partial \upsilon } d\sigma
\right\vert ^{2}dxdt  \notag \\ \notag
&&-\left\langle \left( \int_{\Upsilon }\phi^N
(\cdot ,0;\sigma )d\sigma ,\int_{\Upsilon }\phi^N _{t}(\cdot ,0;\sigma
)d\sigma \right) ,(u^{0,N},u^{1,N})\right\rangle_N\\ \label{def.JN}
&&+\left\langle (\phi
^{0,N},\phi ^{1,N}),(u^{0,N}_T,u^{1,N}_T)\right\rangle_N
\end{eqnarray}%
where $\phi ^{N}$ is the solution of (\ref{eq.23}). Once again, the fact that the solutions of systems (\ref{eq.6}) and (\ref{eq.23}) coincide for initial data in $X^N\times X^N$,  allows us to interpret $J^N$ as the restriction of $J$ to $X^N\times X^N$. In particular, existence of minimizers is guaranteed as soon as the continuous system is controllable in average. The characterization of these minimizers as controls for the finite dimensional system is straightforward, following the proof of Theorem \ref{th_1}. In particular we have the following result: 

\begin{theorem}
Let $\left( u^{0,N},u^{1,N}\right),\;  \left( u^{0,N}_T,u^{1,N}_T\right)\in X^{N}\times X^{N}$ be some discrete data and suppose that $(%
\hat{\phi}^{0,N},\hat{\phi}^{1,N})\in X^{N}\times X^{N}$ is a minimizer of $%
J^N $. If $\hat{\phi}^{N}$ is the corresponding solution of (\ref{eq.23}) with
final data $\left( \hat{\phi}^{0,N},\hat{\phi}^{1,N}\right) $ then $f^N(t)=\eta
(t)\int_{\Upsilon }\left. \frac{\partial \hat{\phi}^{N}}{\partial \upsilon} \right|_{\Gamma_0} d\sigma $ is a control
such that the solution of (\ref{eq.21}) satisfies (\ref{eq.cara_dis_con}). In particular,
\begin{equation*}
\int_{\Upsilon }u^{N}\left( x,T;\sigma \right) d\sigma=u^{0,N}_T ,\qquad \int_{\Upsilon
}u_{t}^{N}\left( x,T;\sigma \right) d\sigma =u^{1,N}_T .
\end{equation*}
\end{theorem}


We now state the convergence result for the discrete controls. 

\begin{theorem}
Let $T>0$ be large enough in order to have averaged observability of the continuos wave equation. For a given initial data and target $(u^0,u^1),(u^0_T,u^1_T)\in L^2(\Omega)\times H^{-1}(\Omega)$, let $f$ be the averaged control of the continuous system provided by the minimizer of $J$, and $f^{N}(x,t)$ be the sequence of averaged controls obtained by minimizing the discrete functional $J^N$ for the initial data and target $(P^Nu^0,P^Nu^1),(P^Nu^0_T,P^Nu^1_T)\in X^N\times X^N$ respectively. Then 
\begin{equation} \label{eq_conv}
f^N \to f , \text{ in } L^2(0,T;\Gamma_0).
\end{equation}
\end{theorem}

\begin{proof}
We show that the numerical approximation can be stated as a Ritz-Galerkin approximation of the variational characterization of the control. The result will follow from the classical convergence result for such approximations of variational problems (see for example \cite{RT}). 

We consider the Hilbert space $V=H^1_0\times L^2$ and the following bilinear form on $V$,
$$
A((\phi^0,\phi^1),(\psi^0,\psi^1))=\int_{0}^{T}\eta (t)\int_{\Gamma _{0}}\int_{\Upsilon}a(\sigma)\frac{%
\partial {\phi}}{\partial \upsilon }d\sigma \int_{\Upsilon}a(\sigma)
\frac{\partial \psi }{\partial \upsilon }d\sigma dxdt,  
$$
where $\phi,\; \psi$ are the solutions of the adjoint system (\ref{eq.6}) with final data $(\phi^0,\phi^1), \; (\psi^0,\psi^1)$ respectively.
We also consider the linear form on $V$,
$$
L((\phi^0,\phi^1))= \left\langle
\left( \int_{\Upsilon }\phi (\cdot ,0;\sigma )d\sigma ,\int_{\Upsilon }\phi
_{t}(\cdot ,0;\sigma )d\sigma \right) ,(u^{0},u^{1})\right\rangle
-\left\langle (\phi ^{0},\phi ^{1}),(u^0_T,u^1_T)\right\rangle .
$$
The minimizer of $J$, $(\hat\phi^0,\hat\phi^1)\in V$ solves the variational equation,
$$
A((\hat\phi^0,\hat\phi^1),(\phi^0,\phi^1))=L(\phi^0,\phi^1), \quad \mbox{ for all } (\phi^0,\phi^1) \in V.
$$

Both, the bilinear form $A$ and the linear one $L$ are continuous on $V$. For the continuity of $A$ we refer to the proof of Theorem \ref{th_2} while the continuity of $L$ is a direct consequence of Lemma \ref{le.ene}. The bilinear form is also coercive, as a consequence of the averaged observability. 

Now  we consider the finite dimensional subspace of $V^N=X^N\times X^N \subset V$. {Note that, by our choice of $X^N$, the space $V^N$ becomes dense in $V$ as $N\to \infty$ in the sense that 
$$
\forall (\phi,\psi)\in V, \quad \lim_{N\to \infty} \inf_{(\phi^N,\psi^N)\in V^N} \| (\phi,\psi) -(\phi^N,\psi^N) \|_V=0. 
$$}
We only have to prove that the minimizer of the discrete functional $J^N$, $(\hat\phi^{0,N},\hat\phi^{1,N})\in V^N$ is solution of 
\begin{equation} \label{eq:var_fun}
A((\hat\phi^{0,N},\hat\phi^{1,N}),(\phi^{0,N},\phi^{1,N}))=L(\phi^{0,N},\phi^{1,N}), \quad \mbox{ for all } (\phi^{0,N},\phi^{1,N}) \in V^N.
\end{equation}

But this is straightforward from the variational characterization of the minimizers for $J^N$, Lemma \ref{le.1disc}, formula (\ref{eq_peN}) and the fact that the solutions of the adjoint systems (\ref{eq.23}) and (\ref{eq.6}) coincide for final data $(\phi^{0,N},\phi^{1,N})\in V^N$.

According to the Ritz-Galerkin convergence result we have the following estimate for the minimizers of $J$ and $J^N$
\begin{equation} \label{eq.smo}
\|(\hat\phi^{0},\hat\phi^{1})-(\hat\phi^{0,N},\hat\phi^{1,N})\|_V \leq C \inf_{(\psi^0,\psi^1)\in V^N} \| (\hat\phi^{0},\hat\phi^{1})-(\psi^0,\psi^1)\|_V, 
\end{equation}
for some constant $C$ independent of $N$. The classical approximation result for spectral projections gives the convergence of the right hand side as $N\to \infty$, and therefore the convergence of minimizers. 

Finally, the $L^2-$convergence of controls is a direct consequence of the continuity of the bilinear form $A$. 
\end{proof}

{
\begin{rem}
Convergence rates in (\ref{eq_conv}) can be easily obtained for smooth initial data. In fact, as pointed out in \cite{Lebeau,EZ2} for a single wave equation one can produce smoother controls for smoother initial data. In particular, with initial data $(u^0,u^1)\in \mathcal{H}^\alpha \times H^{\alpha-1}$ one can modify the functional $J$ to recover minimizers that satisfy $(\hat\phi^{0},\hat\phi^{1})\in \mathcal{H}^{\alpha+1} \times \mathcal{H}^{\alpha}$ (see \cite{EZ2}). We only have to add to the functional $J$ a suitable compact support cutoff function $\tilde \eta(x) \in C^\infty( \Gamma_0)$ to guarantee that the control is compact support both in time and space. Following \cite{EZ2}, it is straightforward to check that the same property holds for averaged controls. In this case the right hand side in (\ref{eq.smo}) can be estimated by 
$$
\inf_{(\psi^0,\psi^1)\in V^N} \| (\hat\phi^{0},\hat\phi^{1})-(\psi^0,\psi^1)\|_V\leq \lambda^{-\alpha}_N \inf_{(\psi^0,\psi^1)\in V^N} \| (\hat\phi^{0},\hat\phi^{1})-(\psi^0,\psi^1)\|_{\mathcal{H}^{\alpha+1} \times \mathcal{H}^{\alpha}},
$$ 
which is of the order $N^{-\alpha/d}$, according to the Weyl asymptotic formula for the Dirichlet Laplacian eigenvalues. Therefore, we obtain
$$
\|f^N-f\|_{L^2(0,T;\Gamma_0)} \leq C N^{-\alpha/d}.
$$
This remark was already observed in \cite{B1} for $\alpha=1$.
\end{rem}
}

{
An analogous convergence result can be obtained for the discrete control problem (\ref{eq.20v2}) following the ideas in \cite{BBL}. To simplify, we restrict ourselves to the case where $\Gamma_0$ is parametrized by a unique function ${\bf r}:(0,1)\to \Gamma_0$. A straightforward calculus shows that in this case we can define a class of controls $(f^{N,M}_j)_{j=1}^M$ in the form  
\begin{equation} \label{eq_fNM}
f^{N,M}_j(t)=  \int_{0}^1 \int_\Upsilon a(\sigma)\frac{%
\partial {\hat \phi^{N,M}}}{\partial \upsilon }({\bf r}(s),t) \; |{\bf r}'(s)|^{1/2} \sin(j\pi s) \; d\sigma \; ds ,  
\end{equation}
where $\hat \phi^{N,M}$ is the solution of the adjoint equation (\ref{eq.6}) with final data $(\hat \phi^{0,N,M},\hat \phi^{1,N,M})$. This final data is the solution of a new variational equation, analogous to (\ref{eq:var_fun}), but replacing the bilinear form $A$ by, 
\begin{eqnarray*}
 A^M((\phi^0,\phi^1),(\psi^0,\psi^1))&=&\int_{0}^{T}\eta (t)\sum_{j=1}^M  \left( \int_{0}^1\int_{\Upsilon}a(\sigma)\frac{%
 \partial {\phi}}{\partial \upsilon }({\bf r}(s),t) |{\bf r}'(s)|^{1/2} \sin(j\pi s)  d\sigma \; dx\right)  
\\
&& \left(\int_{0}^1 \int_{\Upsilon}a(\sigma)
\frac{\partial \psi }{\partial \upsilon }({\bf r}(s),t) |{\bf r}'({\bf r^{-1}})|^{1/2} \sin(j\pi s) d\sigma dx \right) dt. 
\end{eqnarray*}
We prove below the convergence of solutions of both variational equations and we obtain the following result: 
\begin{theorem}Let $(f^{N,M}_j)_{j=1}^M$ be the sequence of controls defined in (\ref{eq_fNM}), $\beta_j$ the functions defined in (\ref{eq:bj}) and $f^N$ the control obtained by minimizing the functional $J^N$. Then,
$$
\left\| \sum_{j=1}^M f^{N,M}_j (t) \beta_j(x) - f^N (x,t)\right\|_{L^2(0,T;\Gamma_0)} \to 0, \mbox{ as $M\to \infty$.} 
$$
\end{theorem}
}

{
\begin{proof} 
Note that it is sufficient to prove that $|{\bf r}'(s)|^{1/2} \int_\Upsilon \frac{\partial {\hat \phi^{N,M}}}{\partial \upsilon }({\bf r}(s),t) d\sigma \to |{\bf r}'(s)|^{1/2} \int_\Upsilon \frac{\partial {\hat \phi^{N}}}{\partial \upsilon }({\bf r}(s),t) d\sigma$ in $L^2((0,1)\times (0,T))$, which is easily deduced from the convergence of the final data $(\hat \phi^{0,N,M}, \hat \phi^{1,N,M}) \to (\hat \phi^{0,N},\hat \phi^{1,N})$ in $H^1_0\times L^2$ as $M\to \infty$ and the direct inequality (\ref{eq:dir_in}).  Thus, we only have to prove that the solutions of the variational equation (\ref{eq:var_fun}) converge to those of the same equation when replacing the bilinear form $A$ by $A^M$. This is a consequence of the following convergence result for the associated bilinear forms, 
$$
\sup_{(\phi^0,\phi^1),(\psi^0,\psi^1) \in H^1_0\times L^2}\frac{| A^M((\phi^0,\phi^1),(\psi^0,\psi^1))-A((\phi^0,\phi^1),(\psi^0,\psi^1))| }{\| (\phi^0,\phi^1) \|_{H^1_0\times L^2}  \| (\psi^0,\psi^1) \|_{H^1_0\times L^2}} \to 0, \mbox{ as $N\to \infty$.}
$$
To prove this result we consider $(\phi^0,\phi^1),\; (\psi^0,\psi^1)$ of norm 1 and observe that, with the parametrization of $\Gamma_0$, 
\begin{equation} \label{eq:au1}
A((\phi^0,\phi^1),(\psi^0,\psi^1))=\int_{0}^{T}\eta (t)\int_{0}^1\int_{\Upsilon}({\bf r}(s),t)a(\sigma)\frac{%
\partial {\phi}}{\partial \upsilon }({\bf r}(s),t)d\sigma \int_{\Upsilon}a(\sigma)
\frac{\partial \psi }{\partial \upsilon }({\bf r}(s),t) d\sigma |{\bf r}'(s)| \;ds\;dt. 
\end{equation}
Now, we write for each $(\phi^0,\phi^1)\in H^1_0\times L^2$,
\begin{eqnarray*}
&& |{\bf r}'(s)|^{1/2}\int_{\Upsilon}a(\sigma)\frac{%
\partial {\phi}}{\partial \upsilon }({\bf r}(s),t)d\sigma =\sum_{j=1}^\infty \alpha_j (\phi) \sin(j\pi s), \\ 
&& \alpha_j (\phi)=\int_0^1 |{\bf r}'(s)|^{1/2}\int_{\Upsilon}a(\sigma)\frac{%
\partial {\phi}}{\partial \upsilon }({\bf r}(s),t)d\sigma \sin(j\pi s) \; ds, 
\end{eqnarray*}
and substitute in (\ref{eq:au1}) to obtain, 
$$
A((\phi^0,\phi^1),(\psi^0,\psi^1))=\int_{0}^{T}\eta (t) \sum_{j=1}^\infty \alpha_j (\phi) \alpha_j (\psi)\; dt. 
$$
Therefore, 
\begin{eqnarray*}
&& | A^M((\phi^0,\phi^1),(\psi^0,\psi^1))-A((\phi^0,\phi^1),(\psi^0,\psi^1))|= \left| \int_0^T \eta (t) \sum_{j=M+1}^\infty \alpha_j (\phi) \alpha_j (\psi)\; dt \right|\\
&& \leq \left(\int_0^T \eta (t) \sum_{j=M+1}^\infty |\alpha_j (\phi)|^2\right)^{1/2}\left(\int_0^T \eta (t) \sum_{j=M+1}^\infty |\alpha_j (\psi)|^2\right)^{1/2} .
\end{eqnarray*}
Here both terms converges to zero as $M\to \infty$ due to the direct inequality (\ref{eq:dir_in}) and the dense approximation in $L^2(0,1)$ of the trigonometric basis $\{\sin{j\pi x}\}_{j=1}^M$, as $M\to \infty$.
\end{proof}
}


\section{Matrix formulation and finite dimensional control} \label{sec6}

In this section we consider some implementation issues related with the finite dimensional approximation of the averaged control given by (\ref{eq.20})-(\ref{eq.21}). There are several ways to do that. For instance one can follow the original method in \cite{gll} where the authors compute the minimizer of the discrete functional $J^N$ with a conjugate gradient algorithm. This method should be adapted to the present setting with the averaged control. A more direct approach is to write a matrix formulation of the discrete problem and compute the control using the explicit expression given for the finite dimensional control theory (see \cite{Zua1}). From the practical point of view this second approach requires more computing time but it has a simpler implementation. In this paper we follow this second approach. 
 We divide this section in two subsections where we consider separately the 1-d and 2-d wave equations respectively.  

\subsection{The 1-d wave equation}

We assume $\Omega =(0,1)$ and that the control acts at the extreme $x=1$. The eigenfunctions of the Laplace operator in this domain are given by $%
w_{i}(x)=\sqrt{2}\sin (i\pi x)$ with $i\in \mathbb{N}$ and the corresponding
eigenvalues are $\lambda _{i}^{2}=i^{2}\pi ^{2}$. Note that $w_{i}$ are
normalized in the $L^{2}(\Omega )$ norm. In particular, in this case $X^{N}$ is the
subspace generated by those eigenfunctions $w_{i}$ with $i\leq N$ and the
dimension of $X^{N}$ is therefore $N$.

The function $h^{N}(x)$, solution of (\ref{eq.20}), can be
computed explicitly, 
\begin{equation}
h^{N}(x)=xf^N(t)
\end{equation}%
and 
\begin{equation*}
P^{N}h^{N}(x)=f^{N}(t) P^Nx=f^{N}(t) \sum_{i=1}^N \frac{(-1)^{i+1}}{i\pi} w_i(x),
\end{equation*}

Let us denote by $\{ u_i^N(t)\}_{i=1}^N$ the components of $u^N$ in the basis of eigenfunctions of the Laplace operator. Then, system (\ref{eq.21}) is equivalent to the following system of equations for these components, 
\begin{equation*}
\left\{ 
\begin{array}{l}
(u_{i}^{N})^{\prime \prime }+a(\sigma )\lambda _{i}^{2}u_{i}^{N}=-\lambda
_{i}^{2}f^{N}(t) \frac{(-1)^{i+1}}{i\pi} , \qquad  i=1,...,N \\ 
u_{i}^{N}(0)=u_{i}^{N,0},\quad (u_{i}^{N})^{\prime }(0)=u_{i}^{N,1}, \qquad  i=1,...,N ,
\end{array}%
\right. 
\end{equation*}
with $u_i^{N,k}=\int_0^1 u^kw_i \; dx$,
that can be written in matrix form as, 
\begin{equation} \label{mat0}
\left\{ 
\begin{array}{l}
U^{\prime \prime }+a(\sigma )DU=DGf^N(t), \\ 
U(0)=U^{0},\quad U^{\prime }\left( 0\right) =U^{1}.%
\end{array}%
\right. 
\end{equation}%
where, 
\begin{equation}
U =(u_{1}^{N},u_{2}^{N},...,u_{N}^{N})^{T}, \quad
D =diag(\lambda _{1}^{N},\lambda _{2}^{N},...,\lambda _{N}^{N}), \quad
G =\left(-\frac{1}{\pi},\frac12,...,\frac{(-1)^N}{\pi N}\right)^{T},
\end{equation}
and the initial data,
\begin{equation}
U^k =(u_{1}^{N,k},u_{2}^{N,k},...,u_{N}^{N,k})^{T}, \quad k=0,1.
\end{equation}
Note that we have removed the dependence on $N$ in the notation to simplify, but all the vectors and matrixes in system (\ref{mat0}) depend on $N$. The components of the target are also written in vector form $(U_T^0,U_T^1)$ where,
\begin{equation}
U_T^k =(u_{T,1}^{N,k},u_{T,2}^{N,k},...,u_{T,N}^{N,k})^{T}, \quad k=0,1.
\end{equation}

To apply the general theory of controllability for finite dimensional systems we write  (\ref{mat0}) as a first order one
\begin{equation}
\left\{ 
\begin{array}{l}
Z^{\prime }+A(\sigma)Z=B(\sigma)f^N(t), \\ 
Z(0)=Z^{0}%
\end{array}%
\right.  \label{ode}
\end{equation}%
with $Z=\left( 
\begin{array}{c}
U\\ U^{\prime } \end{array} \right),$ $A(\sigma )=\left( 
\begin{array}{cc}
0 & I \\ 
a(\sigma )D & 0%
\end{array}%
\right) ,B(\sigma)=\dbinom{0}{a(\sigma)DG},$ $Z^0=\left( 
\begin{array}{c}
U^0\\ U^{1} \end{array} \right),$ and the target $Z_T=\left( 
\begin{array}{c}
U^0_T\\ U^{1}_T \end{array} \right).$

According to the control theory for finite dimensional systems (see \cite{Zua1}), a control that drives the initial data $Z^0$ to the target $Z_T$ is given by 
\begin{equation} \label{eq.ca}
f^N(t)=-\eta (t)\int_{\Upsilon }B^{\top }(\sigma )e^{-tA^{\top }\left( \sigma
\right) }d\sigma \; Q_{T}^{-1} \; \left( \int_{\Upsilon } e^{tA(\sigma)} Z^0 \; d\sigma -Z_T \right),
\end{equation}
where $Q_{T}$ is the  average controllability gramiam (see \cite{Zua1}) given by 
\begin{equation} \label{eq_gr}
Q_{T}=\int_{0}^{T}\eta (t)\int_{\Upsilon }e^{(T-t)A(\sigma )}B(\sigma)\; d\sigma
\int_{\Upsilon }B^{\top }(\sigma) e^{(T-t)A^{\top }(\sigma )}d\sigma dt .
\end{equation}%

The fact that $f^N(t)$ in (\ref{eq.ca}) is a control for (\ref{ode}) can be easily checked. We only have to substitute it in the solution of the system (\ref{ode}), given by 
\begin{equation}
Z(t,\sigma )=e^{A(\sigma )t}Z_{0}+\int_{0}^{t}e^{A(\sigma )(t-s)}B(\sigma) f^N(s)\; ds.
\end{equation}%
Moreover, this control is the one obtained from the minimizer of $J^N$ in (\ref{def.JN}).

The whole process can be implemented with the following algorithm:

\bigskip

{
{\bf Algorithm}
\begin{enumerate} 
\item Choose $N$ (number of Fourier nodes), 
\item Define a spacial mesh $x_i=i \Delta x$, $i=0,...,M_x$, $\Delta x=1/M_x$ and a temporal mesh $t_j=j \Delta t$, $j=0,...,M_t$, $\Delta t=T/M_t$. Choose $M_x,M_t\geq 2N$ to avoid the aliasing phenomenon in the approximation of Fourier modes.
\item Define a parameter mesh $\sigma_l=\sigma_0 + l \Delta \sigma$, $l=0,...,M_\sigma$, $\Delta \sigma=(\sigma_2-\sigma_1)/M_\sigma$.
\item From the initial data $(u^0,u^1)$ at the nodes $x_i$ approximate the vector of Fourier coefficients $Z^0$. If the initial data are piecewise continuous we can simply use the trapezoidal rule to approximate $u^k_i=\int_0^1 u^k(x)w_i(x)\; dx$, for $k=0,1$.  
\item Compute the Grammian $Q_T$ from (\ref{eq_gr}). For each $\sigma_l $, $l=0,...,M_l$, compute the matrixes $A(\sigma_l)$ and $B(\sigma_l)$. Use the matrix exponential (expm in MATLAB) and a trapezoidal rule to approximate the integrals in time and $\sigma$. 
\item Use formula (\ref{eq.ca}) to compute the control: For each time $t_j$, $j=1,...,M_t$,
\begin{enumerate} 
\item Compute $G_1=\int_{\Upsilon } e^{t_jA(\sigma)} Z^0 \; d\sigma -Z_T$
\item Solve the linear system $Q_TG_2=G_1$ to compute $G_2$.
\item Compute $f(t_j)=-\eta (t_j)\int_{\Upsilon }B^{\top }(\sigma )e^{-t_jA^{\top }\left( \sigma
\right) }d\sigma \; G_2$
\end{enumerate}
\end{enumerate}
}

\subsection{The 2-d wave equation}

In this section we consider the numerical approximation of the average
control for the 2-d wave equation in a square domain. We assume $\Omega
=(0,1)\times (0,1)$ and that the control acts in $\Gamma_0=\{1\}\times(0,1)
\cup (0,1)\times \{1\}$. For this problem the averaged control of the wave
equation is not known, for any time $T>0$. Our experiments provide a
numerical evidence that such controllability holds.

The eigenfunctions of the Laplace operator in this domain are given by $%
w_{i,j}(x_1,x_2)=4\sin(i\pi x_1)\sin(j\pi x_2)$ with $(i,j)\in \mathbb{N}%
\times \mathbb{N}$ and the corresponding eigenvalues are $%
\lambda_{ij}^2=(i^2+j^2)\pi^2$. Note that $w_{ij}$ are normalized in the $%
L^2(\Omega)$ norm and that we have used a double index to represent the
eigenpair $(\lambda_{ij},w_{ij})$. Of course this does not affect the
results of the previous section. In particular, in this case $X^N$ is the
subspace generated by those eigenfunctions $w_{ij}$ with $i,j\leq N$ and
the dimension of $X^N$ is therefore $N^2$.

Any function $u\in C([0,T];L^2(\Omega))$ can be written as 
$ u(x_1,x_2,t)=\sum_{i,j=1}^\infty u_{ij}(t) w_{ij}(x_1,x_2),
$
where $u_{ij}(t)$ are the Fourier coefficients, and the projection operator $%
P^N:L^2(\Omega) \to X^N$ is defined as 
\begin{equation*}
P^N u(x_1,x_2,t) = \sum_{i,j=1}^N u_{ij}(t) w_{ij}(x_1,x_2).
\end{equation*}

{
To discretize the control $f$ we divide $\Gamma_0$ in two parts, the right and upper  boundaries that we denote $\Gamma_0^1$ and $\Gamma_0^2$ respectively,  and parametrized by, 
\begin{eqnarray*}
{\bf r}_1: (0,1)\to \Gamma_0^1, \quad {\bf r}_1 (s)=(1,s) \in \Gamma_0^1, \\
{\bf r}_2: (0,1)\to \Gamma_0^2, \quad {\bf r}_2 (s)=(s,1) \in \Gamma_0^2.  
\end{eqnarray*}
}

{
For each basis function $\sin(j\pi s)$ in $L^2(0,1)$ we define 
$$
\beta^k_j(x)= \left\{ 
\begin{array}{ll}
\sin(j\pi {\bf r}_k^{-1}(x)), &  x\in \Gamma_0^k, \\
0, &  x\in \Gamma_0\backslash \Gamma_0^k,
\end{array}
\right. 
\quad k=1,2.
$$
According to (\ref{eq:ap_com}), we look for a discrete control in the form,
$$
f^{N}(x,t)=\sum_{j=1}^M (f_j^{1,N}(t) \beta_j^1(x)+f_j^{2,N}(t) \beta_j^2(x)),  \quad x\in \Gamma_0, \quad t\in(0,T),
$$
for some functions $f^{k,N}_j$. 
We next solve the family of elliptic problems 
$$
\left\{
\begin{array}{ll}
\Delta h =0, &\mbox{ $x$ in $\Omega$,}\\
h=\beta_j^k(x), & \mbox{ $x$ on $\Gamma_0$,}\\
h=0, & \mbox{ $x$ on $\Gamma \backslash \Gamma_0$,}
\end{array} 
\right. \quad k=1,2,\quad j=1,...,M.
$$
This give us a family of functions $h_j^k(x)$ in $\Omega$. We can compute explicitly these functions,
$$
h_j^1(x_1,x_2)=\frac{\sinh (j\pi x_{1})}{\sinh (j\pi )}\sin (j\pi x_{2}), \quad h_j^2(x_1,x_2)=\frac{\sinh (j\pi x_{2})}{\sinh (j\pi )}\sin (j\pi x_{1}), \quad j=1,...,M. 
$$
Now we compute $P^N h_j^k$,
$$
P^N h_j^1 = \sum_{i=1}^N c_{ij} w_{ij}, \quad
P^N h_i^2 = \sum_{j=1}^N d_{ij} w_{ij}, 
$$
where 
\begin{equation*}
c_{ij}=2\int_{0}^{1}\frac{\sinh (i\pi x_{1})}{\sinh (i\pi )}\sin (j\pi
x_{1})\;dx_{1}=\frac{2j(-1)^{j+1}}{(i^2+j^2)\pi}, \quad d_{ij}=c_{ij} .
\end{equation*}%
}
System (\ref{eq.20v2}) for this particular example reads, 
\begin{equation} \label{eq_fc2d}
\left\{ 
\begin{array}{l}
(u_{ij}^{N})^{\prime \prime }+a(\sigma )\lambda _{ij}^{2}u_{ij}^{N}=-a(\sigma )\lambda
_{ij}^{2}(c_{ij}f_{i}^{1,N}+d_{ji}f_{j}^{2,N}), \\ 
u_{ij}^{N}(0)=u_{ij}^{N,0},\quad (u_{ij}^{N})^{\prime }(0)=u_{ij}^{N,1}.%
\end{array}%
\right.
\end{equation}
Here we have taken $M=N$ since for a lower value of $M$ some equations would have zero second hand term, i.e. without control.

Now we write the components $u_{ij}$ in a column matrix as follows,
$$
U=(u_{11},u_{21},...,u_{N1},u_{1,2},u_{22},....,u_{N,2},...,u_{1,N},...,u_{NN})^T.
$$
A straightforward computations allows us to write system (\ref{eq_fc2d}) in the matrix form (\ref{mat0}) for suitable matrixes $D,G$ and a column vector with the Fourier coefficients of the controls 
$$
f^N(t)=(f_1^{1,N},f_2^{1,N},...,f_N^{1,N},f_1^{2,N},f_2^{2,N},...,f_N^{2,N})^T. 
$$
The rest is similar to the 1-d case.    

\section{Numerical experiments}\label{sec7}

In this section we compute the control in average in different situations, for the 1-d and 2-d wave equation in the square. In both cases we use formula (\ref{eq.ca}) for the finite dimensional systems obtained by the projection method described above. 

\subsection{The 1-d wave equation}

We consider in this section four experiments. In the first one we assume that the unknown parameter is in an interval and we obtain the control in average, that we compare with the control at a single value of the parameter. We also check that the Grammiam associated to the discrete problem is uniformly bounded by below as $N$ grows, which provides an evidence of the uniform observability. In the second experiment we analyze the behavior of the control in average for large time. In the third experiment we show the behavior of this control when the length of the interval tends to zero. In particular we observe that the control in average converges to the control of the average parameter. The fourth experiment considers a parameter in the union of two small intervals. We see that the control in average makes the solution $u(x,T)$ to approximate either of two different symmetric profiles. 

\bigskip

{\bf Experiment 1.} We first consider the one-dimensional wave equation with initial position and velocity given by $u^0(x)=1-2|x-1/2|$, $u^1(x)=0$, and the equilibrium target $u^0_T=u_1^T=0$. We take final time $T=2.5$, $a(\sigma)=\sigma$, $\sigma\in[1,2]$, the number of Fourier coefficients is $N=30$ and time step $dt=10^{-2}$. The average in $\sigma$ is computed with the trapezoidal rule and step $d\sigma=10^{-2}$. 

{In Figure \ref{fig1b} we show the control in average and the control of a single realization for comparison. We observe that both seem to have the same regularity as the initial position, i.e. continuous but not $C^1$. As the initial data is in $H^1_0 \times L^2$, the solution must be in the same space and we consider this norm. We show the time behavior of the $H^1_0\times L^{2}$-norm of the solutions to illustrate that, as expected, the averages of controls or the control of the averaged parameter do not produce a control in average.  
Finally, we have plotted the solution at time $t=T$ for a few realizations of  the parameter $\sigma$, with the control in average to illustrate the effect of the control. Note that, even if the average of solutions is zero at any $x\in[0,1]$, the solutions for different values of the parameter may be far from zero. 
In this example the velocity $u_t(x,t)$ becomes a discontinuous function in $x$, for $t>0$, even if we do not apply any control. In fact, we can see the oscillatory behavior of a trigonometric approximation for discontinuous functions in the lower right experiment in Figure \ref{fig1b}, which draws the controlled solutions $u_t(x,T;\sigma)$, for different values of $\sigma$.}

{In table \ref{tab1} we show the behavior of the norm of the control and the average of $u(x,T;\sigma)$ for $1,000$ random values of $\sigma$, when the number of Fourier coefficients $N$ grows. This gives a numerical evidence that the norm of the controls remains bounded and that the action of these controls is closer to the control in average as $N$ grows.}   

\begin{figure}%
\begin{tabular}{cc}
\includegraphics[width=8cm]{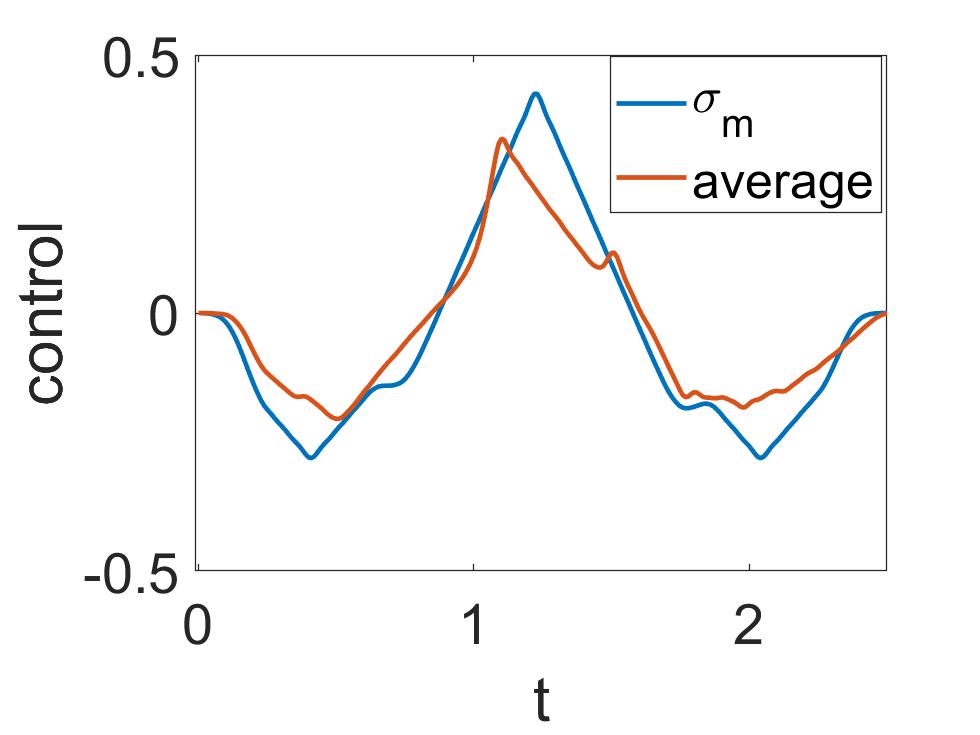}%
& \includegraphics[width=8cm]{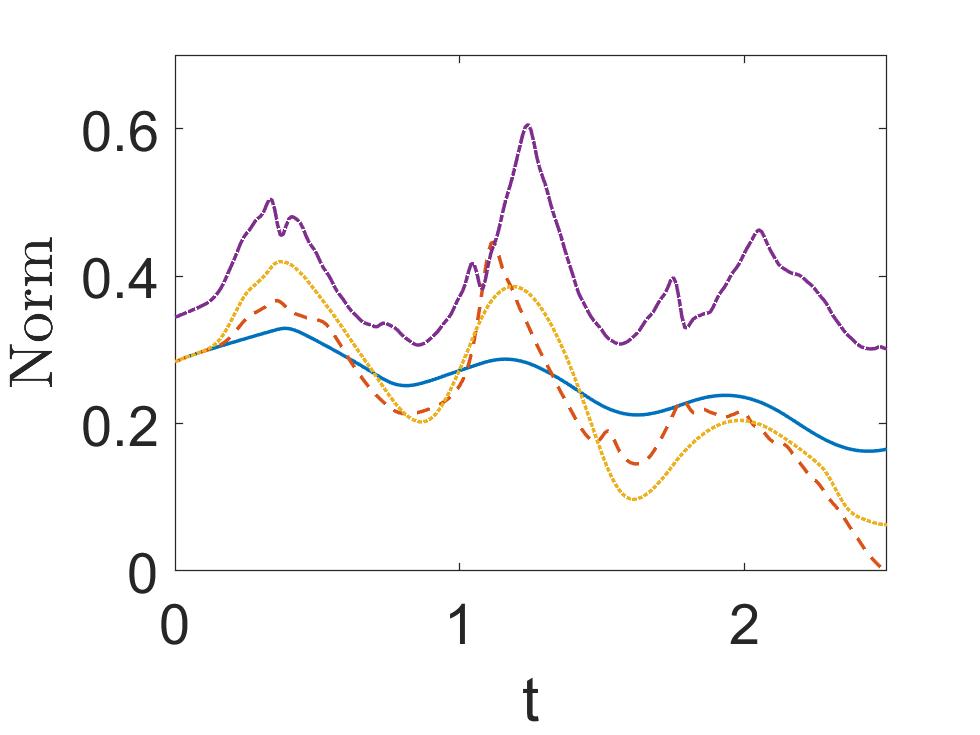}\\
Control in average and control for & Time evolution  of the $H^1_0\times L^{2}$-norm   \\
 the average parameter $\sigma_m=1.5$ &
of the solutions with different controls.
\\
\includegraphics[width=8cm]{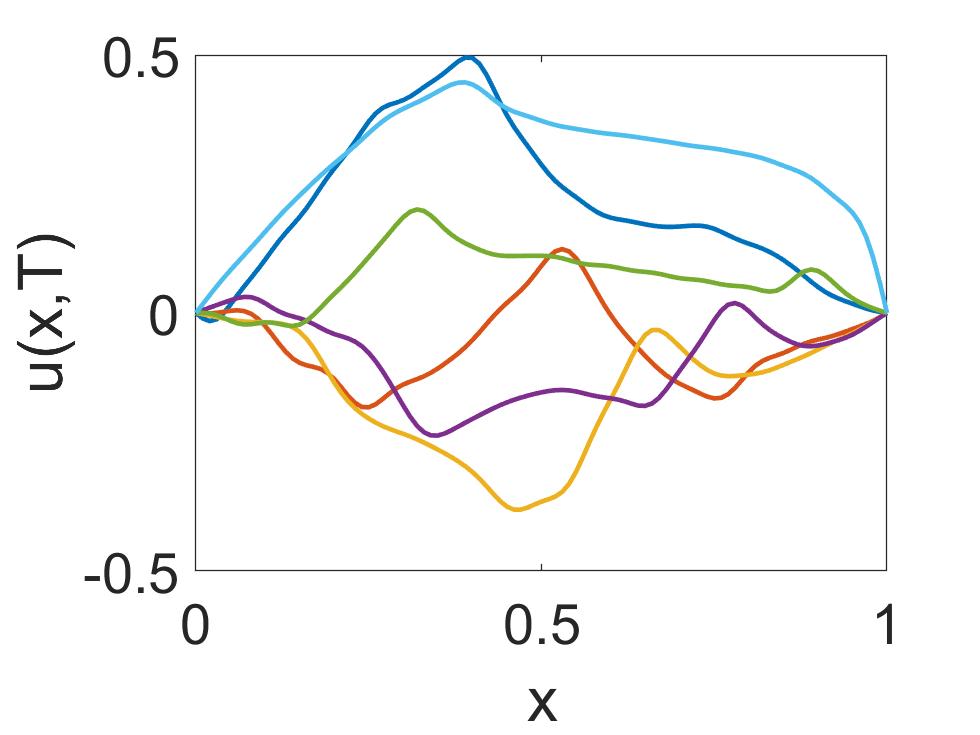}%
& \includegraphics[width=8cm]{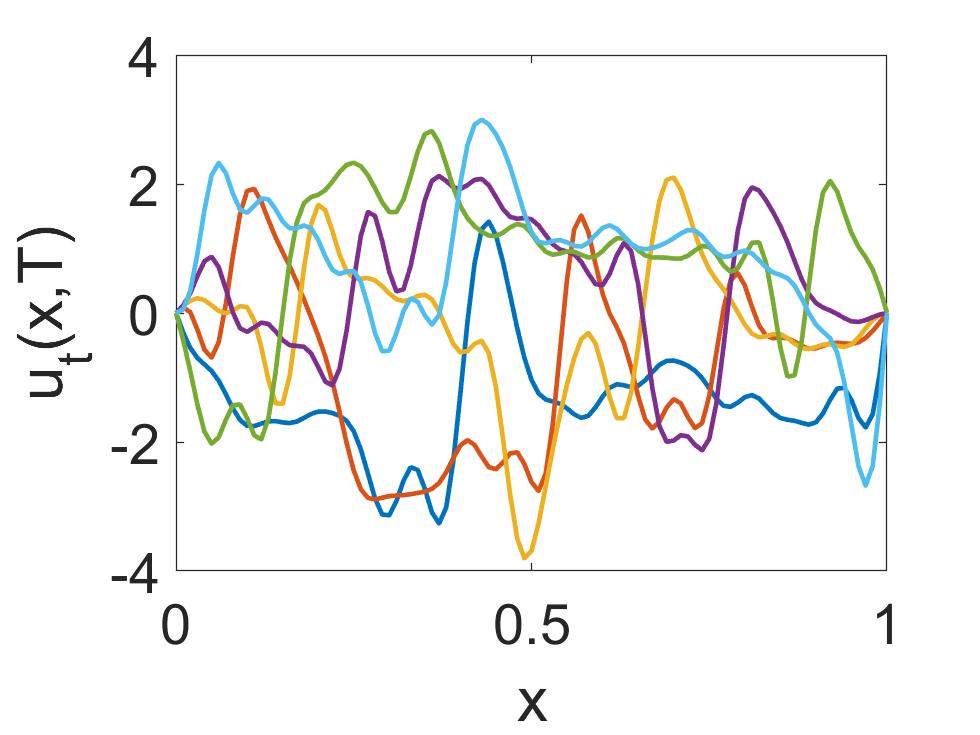}\\
Position of the controlled solutions $u(x,T)$& Velocity of the controlled solutions at $t=T$,    
\end{tabular}
\caption{Numerical results of experiment 1: the value of the unknown parameter $\sigma$ is in the interval $[1,2]$. The upper right figure compares the time evolution of the norm in the following cases: without control (solid blue), with the control in average (dash red), the average of the controls (dot yellow) and the average parameter $\sigma_m=1.5$ (dash-dot cyan).
The two lower simulations contain several controlled solutions at time $t=T$ for different values of the parameter to illustrate that, even if the average is zero in $x\in[0,1]$ when the control is active, the solutions at time $t=T$ may be far from zero. We take $\sigma=1$ (light blue), $1.2$  (dark blue), $1.4$ (green), $1.6$ (red), $1.8$  (cyan) and $2$ (yelow).}%
\label{fig1b}
\end{figure}

\begin{table}
\begin{center}
\begin{tabular}{|c|c|c|}
N & $\| \left( \int_{\Upsilon} u(x,T)d\sigma, \int_{\Upsilon} u_t(x,T)d\sigma \right) \|_{L^2\times H^{-1}}$ & $\| f\|_{L^2}$\\ \hline
$10$ & $2.3421 \times 10^{-4}$ & $2.2243 \times 10^{-1}$\\
$50$ & $3.5533 \times 10^{-5}$ & $2.2245 \times 10^{-1}$\\
$100$ & $8.8591 \times 10^{-6}$ & $2.2245 \times 10^{-1}$
\end{tabular}
\caption{Experiment 1: Norm of the control in average and target as $N$ grows \label{tab1}}
\end{center}
\end{table}

{
We finish this experiment giving a numerical evidence of the uniform bound for the observability inequality for the control in average. We compute the lowest eigenvalue of the Grammian $Q_T$ in (\ref{eq_gr}), $\alpha_1$, which corresponds to the discrete version of the observability inequality (\ref{eq.10eta}). In Figure \ref{fig1b_obs} we show that this lowest eigenvalue remains uniformly bounded as $N$ grows.
}

\begin{figure}%
\begin{center}
\includegraphics[width=8cm]{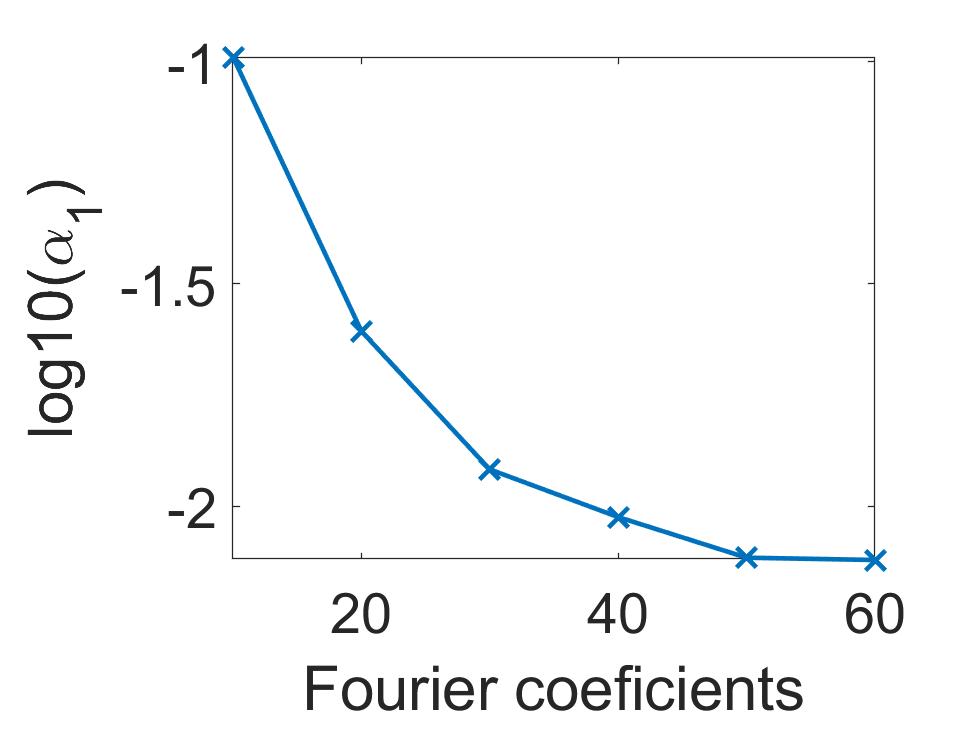}
\caption{Experiment 1: lowest eigenvalue of the discrete Grammiam $Q_T$ in terms of the number of Fourier (log10 scale).}%
\end{center}
\label{fig1b_obs}
\end{figure}

\bigskip

{
{\bf Experiment 2.} In this experiment we compare the behavior of the averaged control with the control corresponding to a single parameter (the average value $\sigma_m$) as $T$ grows. We consider the same initial data and parameters as in experiment 1. In Table \ref{tab_ex2} we compare the $L^2-$norm of both controls. We observe that, as $T$ grows, both the norm of the  averaged control and the one corresponding to a single parameter control slightly decrease. 
}

\begin{table}
\begin{center}
\begin{tabular}{|c|c|c|}
T & $\| f_{ave} \|_{L^2(0,T)}$ & $\| f_{\sigma_m}\|_{L^2}$\\ \hline
$2.5$ & $2.24 \times 10^{-1}$ & $3.09 \times 10^{-1}$\\
$20$ & $1.37 \times 10^{-2}$ & $1.06 \times 10^{-1}$\\
$40$ & $1.07 \times 10^{-2}$ & $7.46 \times 10^{-2}$
\end{tabular}
\caption{Experiment 2: Norm of the control in average and the associated to $\sigma_m=1.5$ as $T$ grows. \label{tab_ex2}}
\end{center}
\end{table}

{
In figure \ref{fig1b_T} we compare the control in average with the control for a single value of the parameter for large time $T=40$. We observe that the control in average acts at the end of the time interval. This is due to the dissipative character of the average evolution. As it is show in the lower simulations, the norm of the average decays, even without any control, to a constant value that depends on the intial data. Therefore, the control in average only have to act at the end of the time interval. 
}

\begin{figure}%
\begin{tabular}{cc}
\includegraphics[width=8cm]{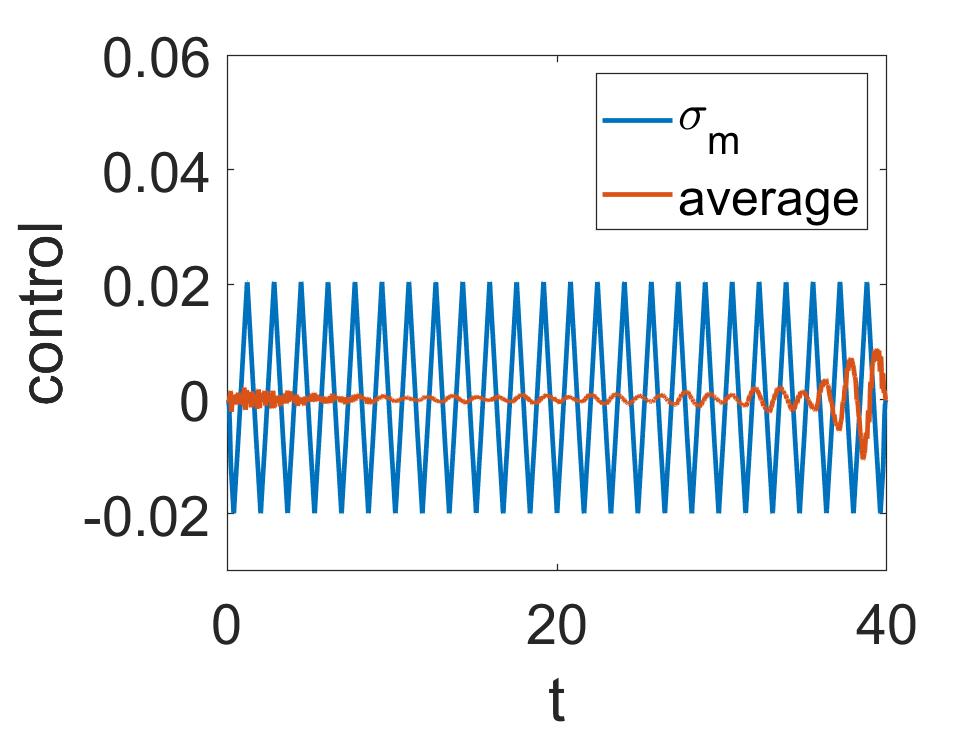}%
& \includegraphics[width=8cm]{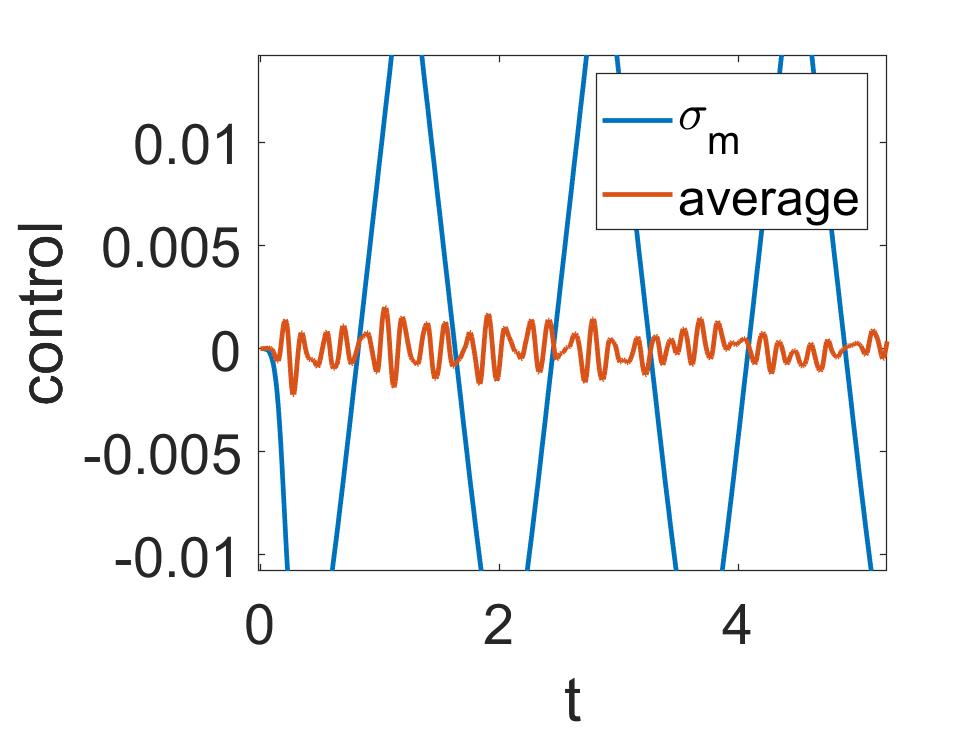}\\
Control in average vs control for $\sigma_m=1.5$& Zoom near $t=0$.\\
\includegraphics[width=8cm]{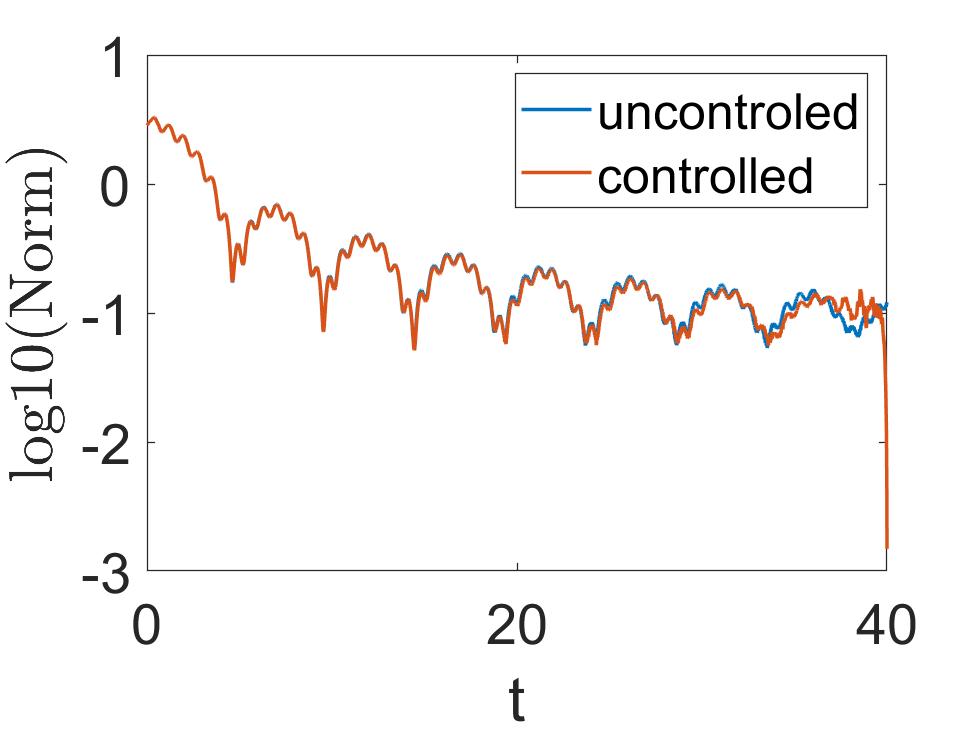}%
& \includegraphics[width=8cm]{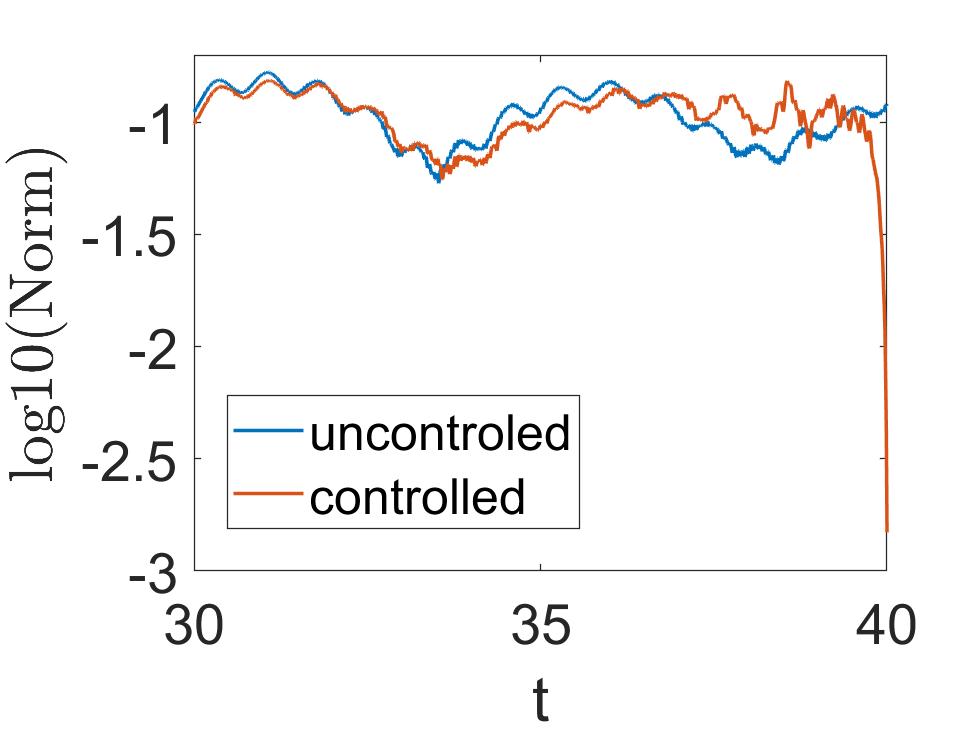}\\
Norm of the solutions & Zoom near $t=T$.
\end{tabular}
\caption{Numerical results of experiment 2: Control in average and the control for a single value of the parameter $\sigma_m=1.5$ (upper left), and a zoom near $t=0$ to illustrate the presence of oscillations in the control in average (upper right). The two low simulations show the time evolution of the $L^2\times H^{-1}$-norm of the average of some solutions (in log scale) when we apply the control in average and without control.}%
\label{fig1b_T}
\end{figure}

\bigskip

{\bf Experiment 3.} Here we illustrate the convergence of the control in average to the control of the average parameter when the set of parameters is an interval with decreasing length. In particular we consider $a(\sigma)=\sigma$, $\sigma \in \Upsilon=[1-\varepsilon,1+\varepsilon]$ for different values of $\epsilon\to 0$ and we compare the control in average with the control for $\sigma_0=1$. The average in $\sigma$ is computed with the trapezoidal rule and step $d\sigma=\varepsilon 10^{-2}$. The rest of the data are as in the experiment 1. In table \ref{tab2} we illustrate the convergence of the control in average to the control of the average parameter.

\begin{table}
\begin{center}
\begin{tabular}{|c|c|}
$\varepsilon$ &  $\| f_{ave} - f_{\sigma_0}\|_{L^2}$\\ \hline
$1/2$ & $2.0006 \times 10^{-1}$ \\
$10^{-1}$ & $2.1701 \times 10^{-2}$ \\
$10^{-2}$ & $3.8824 \times 10^{-4}$ \\
$10^{-3}$ & $3.3179 \times 10^{-6}$ 
\end{tabular}
\caption{Experiment 3: Convergence of the averaged control to the control for the averaged parameter $\sigma_0$ when the lenght of the interval $\varepsilon \to 0$. \label{tab2}}
\end{center}
\end{table}

\bigskip

{\bf Experiment 4.} Now we illustrate the behavior of the control in average when the set of parameters is the union of two intervals. In particular we consider $a(\sigma)=\sigma$, $\sigma \in \Upsilon=[1-\varepsilon,\; 1+\varepsilon]\cup [3/2-\varepsilon,\; 3/2+\varepsilon]$ with $\varepsilon=5\times 10^{-3}$. The average in $\sigma$ is computed with the trapezoidal rule and step $d\sigma=\varepsilon 10^{-2}$. The  rest of the data are as in the Experiment 1. In Figure \ref{fig2} we show the control in average and the controls for the two parameters $\sigma=1, 3/2$. We see that the averaged control exhibits some small oscillations which are not present in the other two controls. We also draw the controlled solutions at $t=T$ for some values of the parameter. We observe how the solution is controlled either to one of two symmetric final states. 

We have also considered the case where one of the intervals contains values of the parameter for which the wave equation is not controllable. In particular, when $\sigma \in \Upsilon=[1/2-\varepsilon,\; 1/2+\varepsilon]\cup [3/2-\varepsilon,\; 3/2+\varepsilon]$ the exact controllability result is not true for $\sigma \in [1/2-\varepsilon,\; 1/2+\varepsilon]$ if we maintain the time $T=2.5$. However, the averaged controllability for $\sigma\in \Upsilon$ seems to hold in this case. 

\begin{figure}%
\begin{tabular}{cc}
\includegraphics[width=8cm]{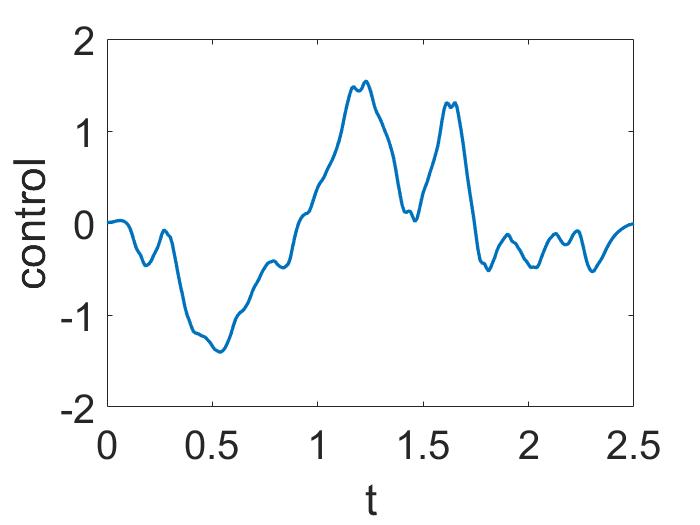}%
& \includegraphics[width=8cm]{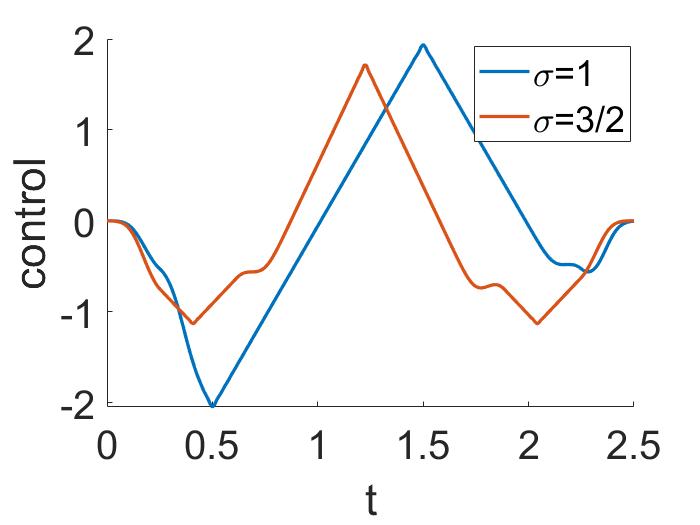}\\
Control in average & Controls for the parameters $\sigma=1,3/2$ \\   
\includegraphics[width=8cm]{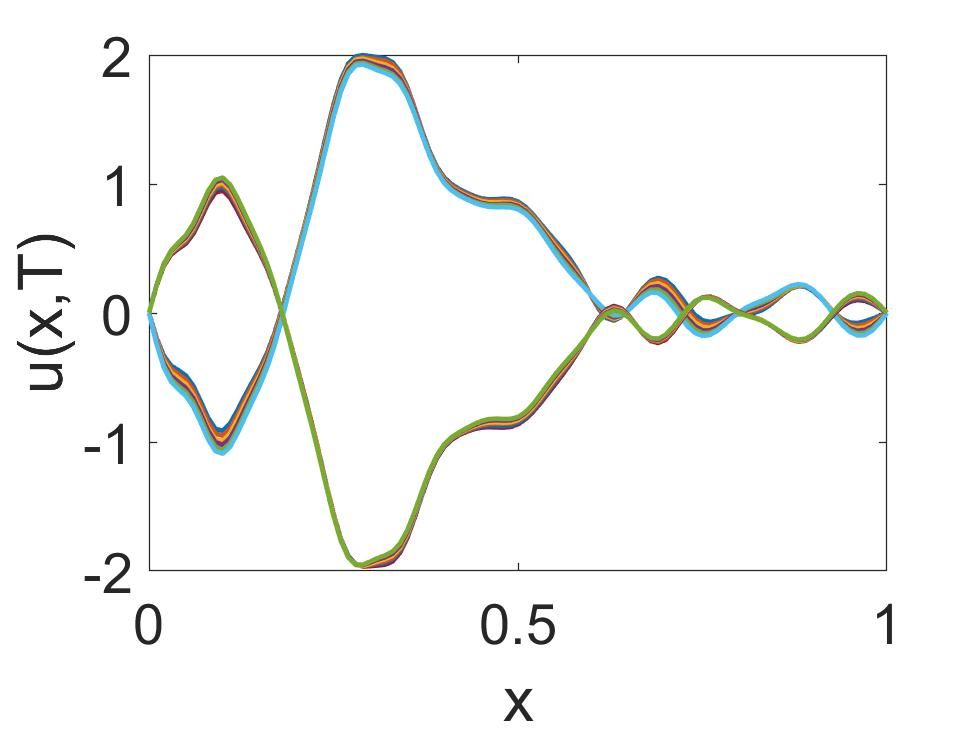}%
& \includegraphics[width=8cm]{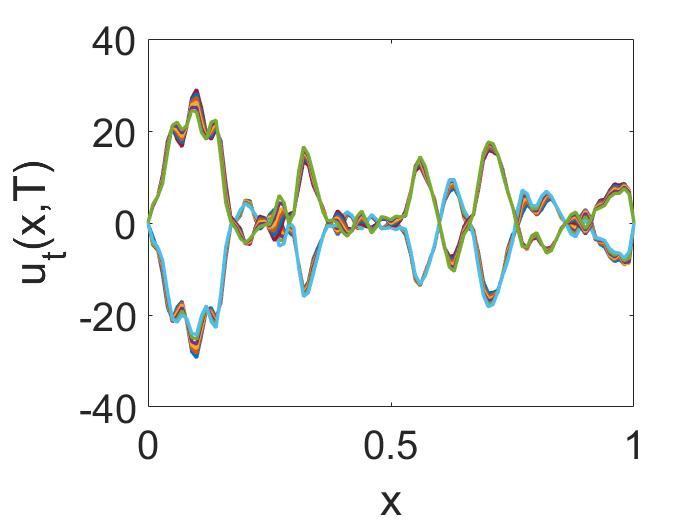}\\
Position for some controlled solutions & Velocity for some controlled solutions
\end{tabular}
\caption{Experiment 4. Here the parameter is equally distributed in two small disjoint intervals around $\sigma=1$ and $\sigma=3/2$. The control chooses to drive the solution to either of two different symmetric profiles, whose average is zero (the chosen target).}%
\label{fig2}%
\end{figure}

\begin{figure}%
\begin{tabular}{cc}
\includegraphics[width=8cm]{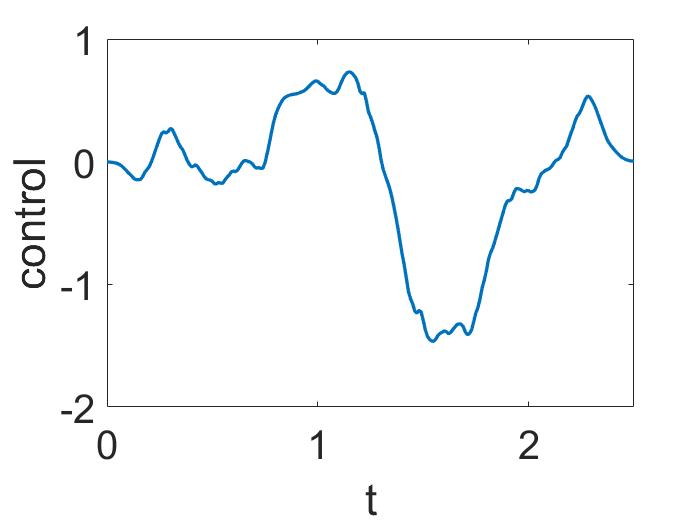}%
& \includegraphics[width=8cm]{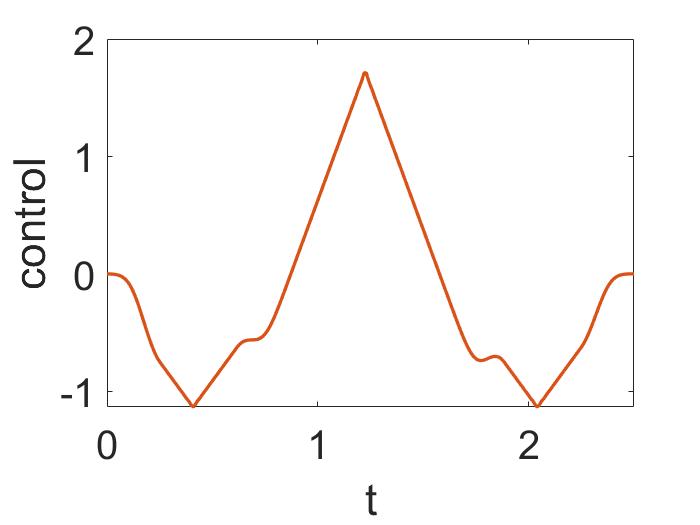}\\
Control in average & Control for the parameter $\sigma=3/2$ \\   
\includegraphics[width=8cm]{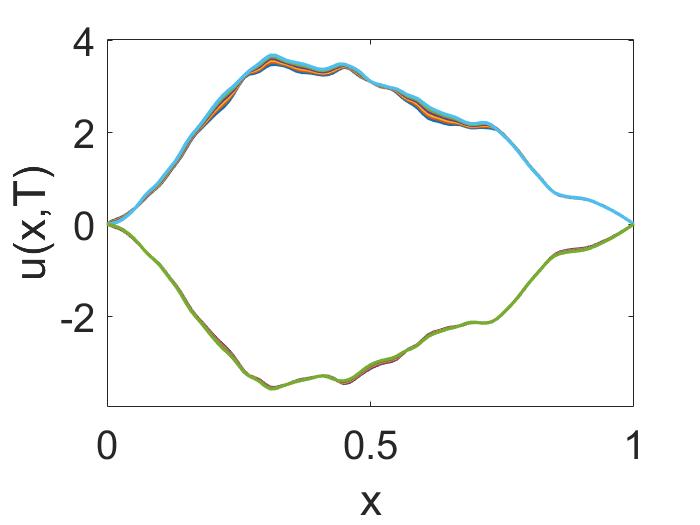}%
& \includegraphics[width=8cm]{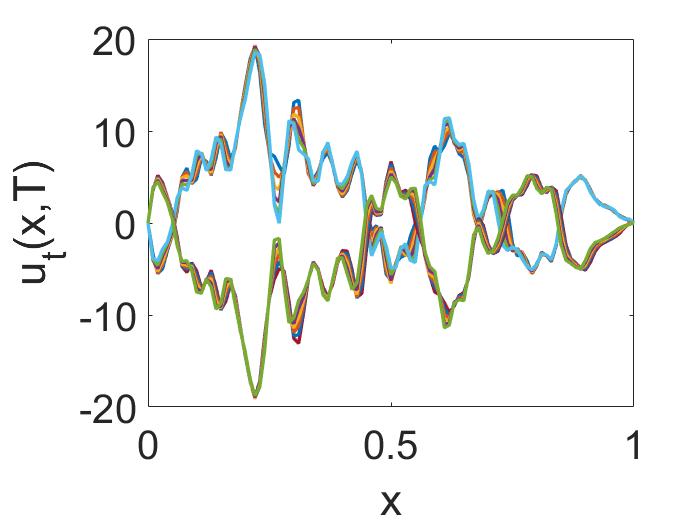}\\
Position for some controlled solutions & Velocity for some controlled solutions
\end{tabular}
\caption{Experiment 4. Here the parameter is equally distributed in two small disjoint intervals around $\sigma=1/2$ and $\sigma=3/2$. The control time is $T=2.5$ and therefore the wave equation is not controllable when $\sigma=1/2$. However, the averaged control exists.}%
\label{fig2n}%
\end{figure}

\subsection{The 2-d wave equation}

In this section we consider a single experiment that illustrate that the method can be easily adapted to multidimensional situations, as long as we know the eigenfunctions of the associated Laplace operator in the domain. 

\bigskip

{\bf Experiment 5.} We consider a square domain $[0,1]\times[0,1]$. The control acts at the two sides $\{1\}\times [0,1] \cup [0,1] \times \{ 1 \}$ in the time interval $t\in [0,2.5]$. We assume $a(\sigma)=\sigma$, $\sigma\in [3/2,2]$ and the average is computed with the trapezoidal rule with step $d\sigma=10^{-2}$. We consider the first $N=100$ Fourier coefficents. The initial data are $u^0=\max \{ x_1,(1-x_1)\} \max \{ x_2,(1-x_2)\} $  and $u^1=0$. The target is the equilibrium $u^0_T=u^1_T=0$. In Figure \ref{fig3} we have drawn the $L^2$-norm of the control acting in the two sides of the square and the evolution of the norm $\|(\int_\Upsilon u \; d\sigma,\int_\Upsilon u_t \; d\sigma)\|_{H^{1}\times L^2}$.  

\begin{figure}%
\begin{tabular}{cc}
\includegraphics[width=8cm]{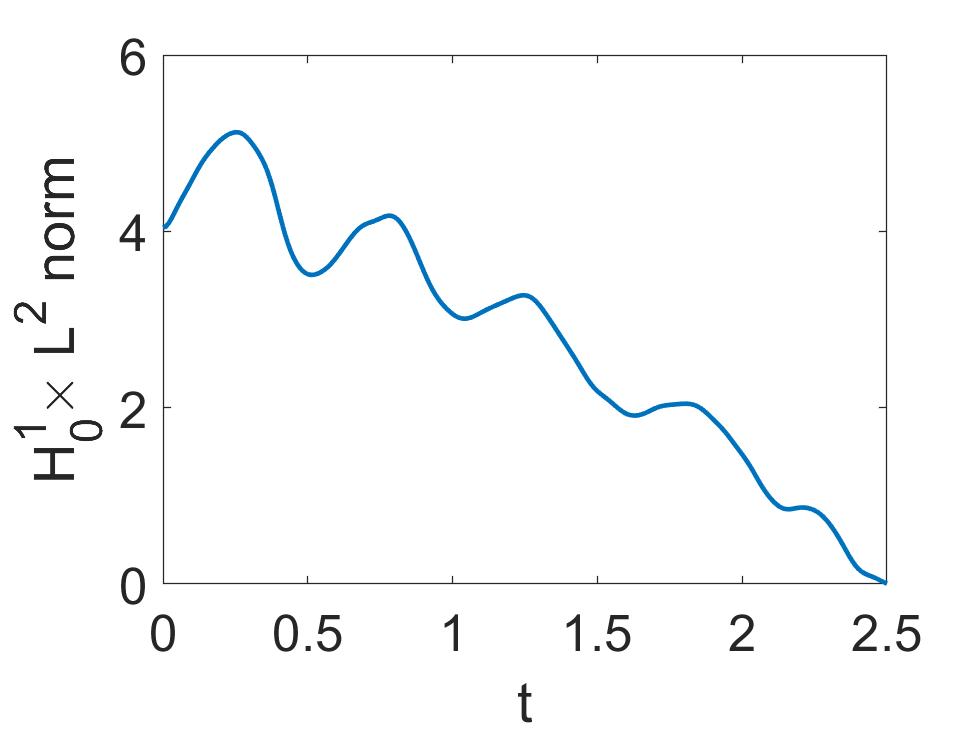}%
& \includegraphics[width=8cm]{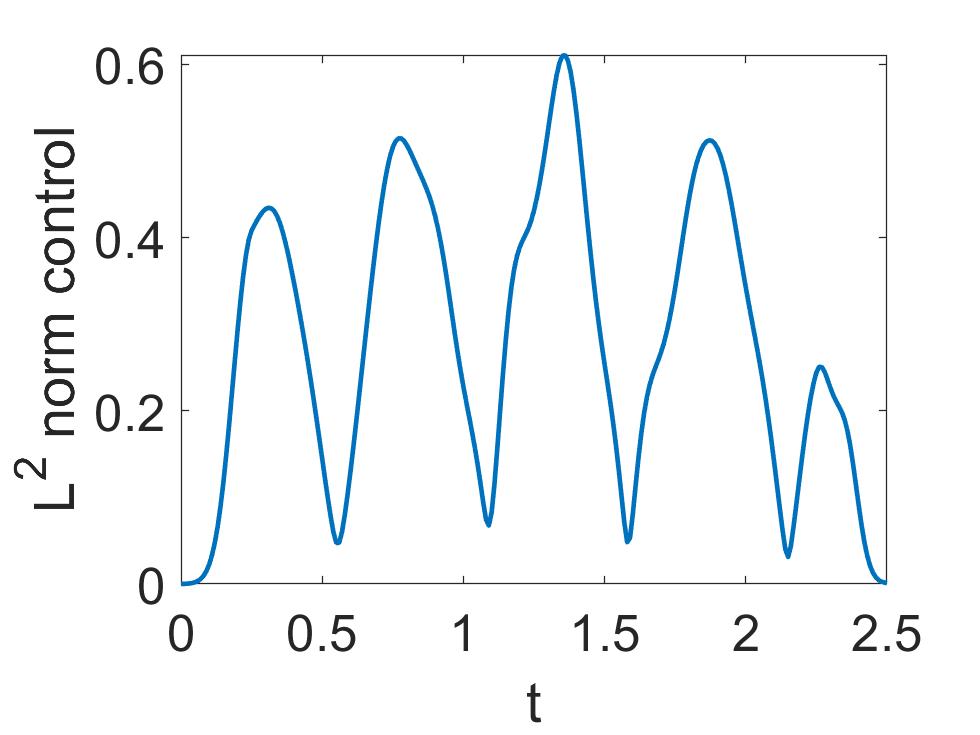}\\
 Time evolution of the norm of $(\int_\Upsilon u \; d\sigma ,\int_\Upsilon u_t \; d\sigma)$ & $L^2$-norm of the control in average  
\end{tabular}
\caption{Experiment 5 corresponding to the boundary control in a square domain. }%
\label{fig3}%
\end{figure}

\section{Conclusions}

We have applied a spectral projection method to approximate the averaged control for the wave equation.  We prove the convergence of the method and how it can be reduced to a classical finite dimensional control problem for which there is an explicit formula for the control. This provides an efficient method to compute averaged controls. We illustrate it with several numerical experiments. The method is quite general and can be easily adapted to similar situations where a variational characterization of controls exists.  

{\bf Acknowledgements.} The first author was supported by the Directorate-General for Scientific Research and Technological Development (DGRSDT) (Algeria). The second author is supported by project MTM2017-85934-C3-3-P from the MICINN (Spain). 

The second author is grateful to Jon Barcena for useful comments on the last version of this paper that contributed to improve it. The authors are also grateful to the anonymous referees for their detailed reading and comments.

\end{document}